%% file: Recursion.tex
\documentclass[11pt,a4paper,reqno]{amsart}

\usepackage{a4wide,amsmath,amssymb,dsfont,epsfig,graphicx,subfig,verbatim,multicol,enumitem,lipsum,amsaddr}
\usepackage[hang,flushmargin]{footmisc}

\usepackage[margin=2.7cm,a4paper]{geometry}

\newtheorem{theorem}{Theorem}
\newtheorem{proposition}[theorem]{Proposition}
\newtheorem{lemma}[theorem]{Lemma}
\newtheorem{corollary}[theorem]{Corollary}
\newtheorem{conjecture}{Conjecture}

\usepackage{pgf,tikz}
\usepackage{tkz-fct}
\usepackage{pgfplots}
\usetikzlibrary{arrows,trees}
\pgfplotsset{ every non boxed x axis/.append style={x axis line style=<->},
     every non boxed y axis/.append style={y axis line style=<->}, every axis/.append style={font=\tiny}}
\definecolor{newpurple}{RGB}{195, 22, 140}
\definecolor{newgray}{RGB}{240, 240, 240}
\definecolor{newlightblue}{RGB}{0, 175, 158}
\definecolor{newblue}{RGB}{47, 50, 145}
\definecolor{newyellow}{RGB}{232, 222, 0}
\definecolor{newgreen}{RGB}{0, 155, 1}

\newcommand{\cS}{\mathcal{S}}
\newcommand{\bd}{\mathbf{d}}
\newcommand{\bq}{\mathbf{q}}
\newcommand{\bP}{\mathbf{P}}
\newcommand{\rR}{\mathrm{R}}
\newcommand{\rT}{\mathrm{T}}
\newcommand{\rL}{\mathrm{L}}
\newcommand{\rS}{\mathrm{S}}
\newcommand{\rA}{\mathrm{A}}
\newcommand{\Mod}{\,\,\mathrm{mod}\,\,}

\newcommand{\diagdots}[3][-25]{%
  \rotatebox{#1}{\makebox[0pt]{\makebox[#2]{\xleaders\hbox{$\cdot$\hskip#3}\hfill\kern0pt}}}%
}

\begin{document}
\title{The prime-power map}
\author{\small Steven}
\address{\normalfont\small Business Mathematics Department, School of Applied STEM, Universitas Prasetiya Mulya, South Tangerang 15339, Indonesia}
\author{\small Jonathan Hoseana}
\address{\normalfont\small Department of Mathematics, Parahyangan Catholic University, Bandung 40141, Indonesia}
\email{steven@pmbs.ac.id\textnormal{, }j.hoseana@unpar.ac.id}
\date{}

\begin{abstract}
We introduce a modification of Pillai's prime map \cite{Pillai,Luca}: the prime-power map. This map fixes $1$, divides its argument by $p$ if it is a prime-power $p^k$, otherwise subtracts from its argument the largest prime-power not exceeding it. We study the iteration of this map over the positive integers, developing, firstly, results parallel to those known for the prime map. Subsequently, we compare its dynamical properties to those of a more manageable variant of the map under which any orbit admits an explicit description. Finally, we present some experimental observations, based on which we conjecture that almost every orbit of the prime-power map contains no prime-power.
\end{abstract}

\maketitle
\section{Introduction}\label{section:Introduction}

Let $\mathbb{P}$ be the set of primes. In 1930, Pillai \cite{Pillai} introduced the \textbf{prime map} $x\mapsto x-\mathbf{p}(x)$, where $\mathbf{p}(x)$ is the largest element of $\mathbb{P}\cup\{0,1\}$ not exceeding $x$, under whose iteration every positive-integer initial condition is eventually fixed at $0$, the only fixed point of the map. Interesting results on the asymptotic behaviour of the time steps $\rR(x)$ it takes for an initial condition $x$ to reach the fixed point were established, before subsequently improved by Luca and Thangadurai \cite{Luca} in 2009. The latter authors proved that $\rR(x)$ grows no faster than $\ln\ln x$ \cite[Theorem 1.1]{Luca}, and that for every $k\in\mathbb{N}$, the proportion of initial conditions $t\leqslant x$ for which $\rR(t)=k$ is asymptotic to $\left(\ln_k x\right)^{-1}$, where the subscript $k$ denotes $k$-fold self-composition \cite[Theorem 1.2]{Luca}.

In this paper we study a modification of the prime map, constructed essentially by letting the role of primes be taken over by prime-powers. A \textbf{prime-power}, as is well-known, is an integer of the form $p^k$, where $p\in\mathbb{P}$ and $k\in\mathbb{N}$. Letting $\mathbb{P}^{\ast}$ be the set of prime-powers, we define the \textbf{prime-power map} $\bP:\mathbb{N}\to\mathbb{N}$ by
$$\bP(x):=\left\{\begin{array}{ll}
1&\text{if }x=1,\\
x-\bq(x)&\text{if }x\notin\mathbb{P}^\ast,\\
\frac{x}{p}&\text{if }x\in\mathbb{P}^\ast\text{ is a power }p\in\mathbb{P},\\
\end{array}\right.$$
where $\bq(x)$ is the largest element of $\mathbb{P}^\ast_{1}:=\mathbb{P}^\ast\cup\{1\}$ not exceeding $x$ (see Figure \ref{fig:plotP}). By the \textbf{orbit} of an \textbf{initial condition} $x\in\mathbb{N}$ under $\bP$ we mean the sequence $\left(x_n\right)_{n=1}^\infty$ where
$$x_1=x\qquad\quad\text{and}\quad\qquad x_{n+1}=\bP\left(x_n\right)\quad\text{for every }n\in\mathbb{N}.$$
Clearly, every orbit under $\bP$ is monotonically non-increasing and eventually reaches the sole fixed point $1$ at which it then stabilises (becomes constant).

In the subdomain $\mathbb{P}^\ast_1\subseteq\mathbb{N}$ the dynamics of $\bP$ is predictable; it consists only of divisions by primes and stabilisation. In its complement, by contrast, the dynamics is non-trivial and resembles that of Pillai's prime map. Such hybrid dynamical behaviour naturally motivates associating to every initial condition $x$ the time steps
$$\rS(x):=\min\left\{n\in\mathbb{N}:x_n\in\mathbb{P}^\ast_1\right\}\qquad\text{and}\qquad \rT(x):=\min\left\{n\in\mathbb{N}:x_n=1\right\}$$
at which the orbit of $x$ enters the predictable subdomain $\mathbb{P}^\ast_1$, and at which it reaches $1$, respectively. These will be referred to as the \textbf{settling time} and \textbf{transient length} of $x$, respectively. The unboundedness of the transient length function is immediate ---the function diverges along the sequence of powers of any given prime--- whereas that of the settling time function is less obvious and will be established in this paper. We shall also be interested in the \textbf{attractor}
$$\rA(x):=x_{\rS(x)}$$
of $x$, i.e., the element of the predictable subdomain first visited by the orbit of $x$.

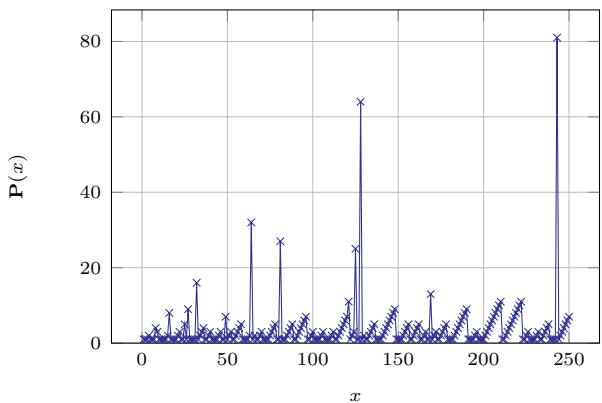
\begin{figure}[h!]
\centering
\input{fig-plotP}
\caption{\label{fig:plotP}\small
The values of $\bP(x)$ for $x\in\{1,\ldots,250\}$.}
\end{figure}

The functions $\rT$ and $\rS$ defined above constitute our primary object of study. Our main results, along with the structure of this paper, will now be described. Following this introduction is Section \ref{sec:TandS}, where we establish \textit{a logarithmic upper bound for $\rT$} (Lemma \ref{lemma:Bertrand2} and Theorem \ref{thm:power2}). Subsequently, improved versions of Lemma \ref{lemma:Bertrand2}, derived from improved versions of Bertrand's Postulate currently known in the literature, will be discussed (Propositions \ref{prop:epsilon} and \ref{prop:Bertrandrefinement}) before we turn our attention to the function $\rS$, establishing \textit{its unboundedness} (Theorem \ref{thm:unboundednessS}) and \textit{a double-logarithmic upper bound} (Theorem \ref{thm:orderS(x)}); these are analogues of the results in \cite[pages 160-161]{Pillai} and \cite[Theorem 1.1]{Luca}, respectively. Next, we will study the densities of some subsets of $\mathbb{N}$ induced by $\bP$ and prove that  \textit{almost every positive integer is coprime with its image under $\bP$} (Theorem \ref{thm:coprime}). At the end of the section, we connect the map $\bP$ to the representations of integers as sums of prime-powers; of note is the discussion on the smallest initial condition having a given settling time (Proposition \ref{prop:xis}).

The detailed behaviour of the functions $\rT$ and $\rS$ is difficult to explicitise. For a comparison, in Section \ref{sec:Pp} we introduce a more manageable variant of $\bP$: a family of maps $\bP_p$, under which the analogous functions $\rT_p$ and $\rS_p$ admit explicit descriptions (Theorem \ref{thm:Fr}). We prove that these two functions are asymptotic (Corollary \ref{corollary}), which, together with some experimental observations, suggests that the original functions $\rT$ and $\rS$ exhibit an analogous asymptoticity. Motivated by these observations, we conjecture that \textit{almost every orbit of the prime-power map contains no prime-power} (Conjecture \ref{finalconjecture}).\bigskip

\section{Results on the map $\bP$}\label{sec:TandS}

Before entering into our main results, let us state some fundamentals. First, we state the famous \textbf{Bertrand's Postulate} \cite[Theorem 418]{HardyWright} which will be used frequently.\bigskip

\begin{lemma}[Bertrand's Postulate]\label{lemma:Bertrand}
	 For every integer $n\geqslant 2$, there exists a prime $p$ such that $n <p< 2n$.
\end{lemma}\bigskip

We will also use standard asymptotic notations which, to ensure clarity, will now be explained. Let $f(x)$ and $g(x)$ be positive for all sufficiently large values of $x\in\mathbb{N}$. We write $f(x)\sim g(x)$ to mean $f(x)$ \textbf{is asymptotic to} $g(x)$, i.e., $\frac{f(x)}{g(x)}\xrightarrow{x\to\infty}1$. In addition, $f(x)\ll g(x)$ means $f(x)$ \textbf{is asymptotically bounded above by} $g(x)$, i.e., there exist $C>0$ and $x_0\in\mathbb{N}$ such that $f(x)\leqslant C\,g(x)$ for every integer $x\geqslant x_0$. Finally, $f(x)=o(g(x))$ means $\frac{f(x)}{g(x)}\xrightarrow{x\to\infty}0$.

\subsection{The transient length function}

First, we have
$$\limsup_{x\to\infty}\rT(x)=\infty,$$
since $\rT$ diverges along, e.g., the sequence of powers of two. More precisely, if $x$ is a power of two, we have $\rT(x)= \log_2x +1$. Let us now prove that $\log_2x +1$ is an upper bound for $\rT$ (see Figure \ref{fig:transittime}) via the following lemma which exploits Bertrand's Postulate.\bigskip

\begin{lemma}\label{lemma:Bertrand2}
For every integer $x\geqslant 2$ we have $$\bP(x)\leqslant\frac{x}{2}.$$
\end{lemma}
\begin{proof}
Let $x\geqslant 2$. If $x=p^k$ for some $p\in\mathbb{P}$ and $k\in\mathbb{N}$, then $\bP(x)=\frac{x}{p}\leqslant\frac{x}{2}$. Otherwise, since Bertrand's Postulate guarantees the existence of a prime between $\frac{x}{2}$ and $x$, then $\bq(x)\geqslant\frac{x}{2}$, and hence $\bP(x)\leqslant\frac{x}{2}$.
\end{proof}\bigskip

\begin{theorem}\label{thm:power2}
	For every $x\in\mathbb{N}$ we have $\rT(x)\leqslant \log_2x +1$ with equality if and only if $x$ is a power of two.
\end{theorem}
\begin{proof}
	By Lemma \ref{lemma:Bertrand2}, in the orbit $\left(x_n\right)_{n=1}^\infty$ of $x\in\mathbb{N}$ we have
$$1=x_{\rT(x)}\leqslant \frac{x_{\rT(x)-1}}{2}\leqslant\frac{x_{\rT(x)-2}}{2^2}\leqslant\cdots\leqslant\frac{x_1}{2^{\rT(x)-1}}=\frac{x}{2^{\rT(x)-1}},$$
which implies the desired inequality. If $x$ is a power of two, we have $\rT(x)=\log_2 x + 1$ as previously remarked. Conversely, if $\rT(x)=\log_2 x + 1$, then $x=2^{\rT(x)-1}$.
\end{proof}\bigskip

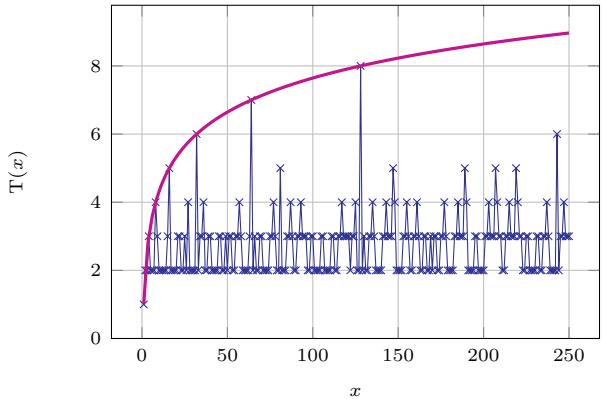
\begin{figure}
\centering
\input{fig-transittime}
\caption{\label{fig:transittime}\small
The values of $\rT(x)$ for $x\in\{1,\ldots,250\}$ and the upper bound provided by Theorem \ref{thm:power2}.}
\end{figure}

In the literature there have been numerous refinements of Bertrand's Postulate which lead to improved versions of Lemma \ref{lemma:Bertrand2}. For instance, the Prime Number Theorem implies that, given any $\varepsilon>0$, there exists $X_0\in\mathbb{N}$ such that for every integer $x\geqslant X_0$, the interval $\left[x-\varepsilon x,x\right]$ contains a prime number \cite[Section 22.19]{HardyWright}, implying that $\bP(x)\leqslant\varepsilon x$ if $x$ is not a prime-power. Therefore, we have the following proposition.\bigskip

\begin{proposition}\label{prop:epsilon}
For every $\varepsilon>0$ there exists $X_0\in\mathbb{N}$ such that
\begin{equation}\label{eq:epsilon}
\bP(x)\leqslant\varepsilon x
\end{equation}
for every non-prime-power $x\geqslant X_0$.
\end{proposition}\bigskip

Notice that, since for every prime-power $x=p^k$ we have $\frac{\bP(x)}{x}=\frac{1}{p}\leqslant\frac{1}{2}$ with equality if and only if $p=2$, there is no $\varepsilon<\frac{1}{2}$ such that \eqref{eq:epsilon} holds for \textit{all} sufficiently large $x$. In this sense, the non-universality of \eqref{eq:epsilon} ---the fact that it holds only for non-prime-powers--- is unavoidable. However, for every $\varepsilon>0$ one can always choose a sufficiently large $X_1\in\mathbb{N}$ such that \eqref{eq:epsilon} holds for all powers of all primes $p\geqslant X_1$.

Hoheisel \cite{Hoheisel}, as cited in \cite[Section 2]{Luca}, proved the following development of Bertrand's Postulate.\bigskip

\begin{lemma}\label{lemma:Hoheisel}
There exist $\theta\in(0,1)$ and $X_0\in\mathbb{N}$ such that for every $x\geqslant X_0$ the interval $\left[x-x^\theta,x\right]$ contains a prime number.
\end{lemma}\bigskip

\noindent This lemma immediately implies the following proposition.\bigskip

\begin{proposition}\label{prop:Bertrandrefinement}
There exist $\theta\in(0,1)$ and $X_0\in\mathbb{N}$ such that
$$\bP(x)\leqslant x^{\theta}$$
for every non-prime-power $x\geqslant X_0$.
\end{proposition}\bigskip

\noindent The authors of \cite{Baker} have proved the existence of $X_0\in\mathbb{N}$ for which one could take $\theta=0.525$.\bigskip

\subsection{The settling time function}

In the first half of this subsection we follow the steps of Pillai's proof in \cite[pages 160-161]{Pillai} of the unboundedness of $\rR$. In the work cited, Pillai proved that
$$\limsup_{x\to\infty}\rR(x)=\infty$$
using a well-known result on the unboundedness of gaps between primes, i.e., that for every $n\in\mathbb{N}$, none of the $n-1$ consecutive positive integers
$$n!+2,\,\,\,n!+3,\,\,\,\ldots,\,\,\,n!+n$$
is prime, implying the existence of two consecutive primes $p_1$ and $p_2$ such that $p_2-p_1\geqslant n$ \cite[Lemma 1]{Pillai}. Here we first modify this argument to show the unboundedness of gaps between prime-powers. Then we will use the result to prove that
$$\limsup_{x\to\infty}\rS(x)=\infty.$$\smallskip

\begin{lemma}\label{lemma:PrimePowerGap}
	For every $n\in\mathbb{N}$ there exist two consecutive prime-powers $q_1$ and $q_2$ such that $q_2-q_1\geqslant n$.
\end{lemma}
\begin{proof}
    Let $n\in\mathbb{N}$. If $n=1$, clearly the lemma holds. Now, suppose $n\geqslant 2$. Let $p_1$, \ldots, $p_m$, where $m\in\mathbb{N}$, be all prime numbers not exceeding $n$. For every $i\in\{1,\ldots,m\}$, define
    $$k_i=\max\{k\in\mathbb{N}:p_i^k\leqslant n\}.$$
    Furthermore, let
    $$\mathcal{N}:=\prod_{i=1}^{m} p_i^{k_i+1}.$$
    
    We will now prove that none of the $n-1$ consecutive positive integers
    $$\mathcal{N}+2,\,\,\,\mathcal{N}+3,\,\,\,\ldots,\,\,\,\mathcal{N}+n$$	
	is a prime-power. Let $i\in\{2,\ldots,n\}$. Consider the number $\mathcal{N}+i$ which is divisible by $i$. If $i$ is not a prime-power, this immediately implies that $\mathcal{N}+i$ is not a prime-power. If $i$ is a prime-power, say $p^\ell$, then $\mathcal{N}+i=\mathcal{N}+p^\ell$ is a number greater than $p^\ell$ which is not divisible by $p^{\ell+1}$ (since $\mathcal{N}$ is divisible by $p^{\ell+1}$ but $p^\ell$ is not), and so it has a prime factor other than $p$, meaning that it is not a prime-power.
	
	Now, let $q_1$ be the largest prime-power less than $\mathcal{N}+2$ and $q_2$ be the smallest prime-power greater than $\mathcal{N}+n$. Then $$q_2-q_1\geqslant (\mathcal{N}+n+1)-(\mathcal{N}+1)\geqslant n,$$
	proving the lemma.
\end{proof}\bigskip

\begin{theorem}\label{thm:unboundednessS}
	We have
	$$\limsup_{x\to\infty} \rS(x)=\infty.$$
\end{theorem}
\begin{proof}
The theorem follows if we can prove that for every non-prime-power $x\in\mathbb{N}$ there exists a non-prime-power $y\in\mathbb{N}$ such that $\rS(y)=\rS(x)+1$. Let $x\in\mathbb{N}$ be a non-prime-power. By Lemma \ref{lemma:PrimePowerGap}, there exist two consecutive prime-powers $q_1$ and $q_2$ such that $q_2-q_1\geqslant x+1$, i.e., $q_1+x+1\leqslant q_2$. Now, pick $y=q_1+x$. Then $q_1<y<q_2$, so $y$ is not a prime-power and the largest prime-power not exceeding $y$ is $q_1$. Therefore, $\bP(y)=y-q_1=x$, meaning that $\rS(y)=\rS(x)+1$, as desired.
\end{proof}\bigskip

Luca and Thangadurai proved that $\rR(x)\ll \ln\ln x$ \cite[Theorem 1.1]{Luca}. Now, we will prove that the same asymptotic upper bound also applies for the settling time function $\rS$, following the argument in the work cited. First, we take numbers $X_0\in\mathbb{N}$ and $\theta\in(0,1)$ which satisfy Lemma \ref{lemma:Hoheisel} (and hence Proposition \ref{prop:Bertrandrefinement}). Then we have the following lemma.\bigskip

\begin{lemma}\label{lemma:orderS(x)}
Let $x\geqslant X_0$, and let $\left(x_n\right)_{n=1}^\infty$ be the orbit of $x$. Then for every $k\in\mathbb{N}$ the following holds:
$$\text{If }x_k\geqslant X_0\text{ and }\rS(x)\geqslant k+1,\text{ then }x_{k+1}\leqslant x^{\theta^k}.$$
\end{lemma}
\begin{proof}
We use induction on $k\in\mathbb{N}$. For $k=1$ the statement reads
$$\text{If }x_1\geqslant X_0\text{ and }\rS(x)\geqslant 2,\text{ then }x_{2}\leqslant x^{\theta},$$
which is true since, if $x_1\geqslant X_0$ and $\rS(x)\geqslant 2$, the latter implying that $x_1$ is not a prime-power, then $$x_2=\bP\left(x_1\right)\leqslant {x_1}^\theta=x^\theta,$$ by Proposition \ref{prop:Bertrandrefinement}.

Now let $k\geqslant 2$, and suppose the statement holds for $k-1$ replacing $k$. Suppose $x_k\geqslant X_0$ and $\rS(x)\geqslant k+1$. The former implies that $x_{k-1}\geqslant X_0$ since $x_{k-1}\geqslant x_k$, while the latter implies that $\rS(x)\geqslant k$ and that $x_{k}$ is not a prime-power. Therefore, our inductive hypothesis gives $x_k\leqslant x^{\theta^{k-1}}$, and so by Proposition \ref{prop:Bertrandrefinement},
$$x_{k+1}=\bP\left(x_k\right)\leqslant {x_k}^\theta\leqslant \left(x^{\theta^{k-1}}\right)^\theta=x^{\theta^k},$$
completing the induction.
\end{proof}\bigskip

\begin{theorem}\label{thm:orderS(x)}
We have
$$\rS(x)\ll\ln\ln x.$$
\end{theorem}
\begin{proof}
Define $X_0':=\max\left\{X_0,e\right\}$. Let $x\geqslant X_0'$, and let $\left(x_n\right)_{n=1}^\infty$ be the orbit of $x$. Define
$$k:=\min\{\rS(x),\ell\}-1\in\mathbb{N}_0,\qquad\text{where}\qquad \ell:=\max\left\{\lambda\in\mathbb{N}:x_\lambda\geqslant X_0'\right\},$$
the existence of $\ell$ being guaranteed by the fact that $x_1=x\geqslant X_0'>1$.

By the definition of $k$ we have that $k\leqslant \rS(x)-1$ and $k\leqslant\ell-1$ (i.e., $k+1\leqslant\ell$), implying $\rS(x)\geqslant k+1$ and $x_k\geqslant x_{k+1}\geqslant x_{\ell}\geqslant X_0'\geqslant X_0$, respectively. Then we have the inequality
$$
x_{k+1}\leqslant x^{\theta^k},
$$
which is trivially true if $k=0$ and implied by Lemma \ref{lemma:orderS(x)} if $k\geqslant 1$. This inequality, together with the fact that $x_{k+1}\geqslant X_0'$, implies
$$X_0'\leqslant x^{\theta^k}.$$
Taking logarithms twice gives
$$\ln \ln X_0'\leqslant k \ln\theta + \ln\ln x.$$
Since $X_0'\geqslant e$, the left-hand side is non-negative, and hence so is the right-hand side, i.e.,
$$k\leqslant\frac{\ln\ln x}{\ln\frac{1}{\theta}}.$$
Letting $m:=\max\left\{\rS(1),\ldots,\rS\left(X_0'\right)\right\}$, by the definitions of $k$ and $\ell$ we have
$$\rS(x)\leqslant k+1+m\leqslant\frac{\ln\ln x}{\ln\frac{1}{\theta}}+1+m\ll \ln \ln x,$$
proving the theorem.
\end{proof}\bigskip


\subsection{Densities of subsets of $\mathbb{N}$ induced by $\bP$} Let us now discuss the densities of some positive integer sets induced by $\bP$, for which the following standard terminology \cite{Sonnenschein,Tenenbaum} will be used. Given a subset $A\subseteq\mathbb{N}$. For every $x\in\mathbb{N}$, we define
$$\bd_x A:=\frac{|A\cap\{1,\ldots,x\}|}{x}.$$
The (\textbf{natural}) \textbf{density} of $A$ is the limit
$$\bd A:=\lim_{x\to\infty}\bd_x A,$$
if it exists. If $\bd A=0$, we say that $A$ is \textbf{sparse}. If $\bd A=1$, we say that $A$ has \textbf{full density}. If $\bd A=1$ and a statement $\mathcal{P}(x)$ holds for every $x\in A$, we say that $\mathcal{P}(x)$ holds for \textbf{almost every} $x\in\mathbb{N}$.

It is fairly well known that the sets $\mathbb{P}$ and $\mathbb{P}^\ast$ are both sparse. This can be deduced by considering the prime and prime-power counting functions
$$\pi(x):=\left|\mathbb{P}\cap\{1,\ldots,\lfloor x\rfloor\}\right|\qquad\text{and}\qquad\Pi(x):=\left|\mathbb{P}^\ast\cap\{1,\ldots,\lfloor x\rfloor\}\right|$$
defined for real numbers $x>0$. Clearly, $\pi(x)\leqslant\Pi(x)$ for every $x>0$. The \textbf{Prime Number Theorem} \cite[page 10]{HardyWright} gives the asymptotic
$$
\pi(x)\sim\frac{x}{\ln x},
$$
from which it follows that $\mathbb{P}$ is sparse. Moreover, since
$$\Pi(x)-\pi(x)=\sum_{j=2}^{\left\lfloor\log_2 x\right\rfloor}\pi\left(x^{\frac{1}{j}}\right)\ll\frac{x^{\frac{1}{2}}}{\ln x^{\frac{1}{2}}}\left(\left\lfloor\log_2 x\right\rfloor-1\right)\ll x^{\frac{1}{2}},$$
then $\Pi(x)\sim\pi(x)$, implying that $\mathbb{P}^\ast$ is also sparse. Some immediate consequences are:\smallskip
\begin{enumerate}\setlength{\itemsep}{4pt}
\item[i)] The set $\rT^{-1}\left(\mathbb{N}_{\geqslant 3}\right)$ has full density because $\rT^{-1}(\{1\})=\{1\}$ and $\rT^{-1}(\{2\})=\mathbb{P}\cup\left(\mathbb{P}^\ast+1\right)$ are both sparse\footnote{For every $k\in\mathbb{N}$, we denote by $\mathbb{N}_{\geqslant k}$ the set of positive integers greater than or equal to $k$. Also, $\mathbb{P}^\ast+1:=\left\{q+1:q\in\mathbb{P}^\ast\right\}$.}.
\item[ii)] The set $\rS^{-1}\left(\mathbb{N}_{\geqslant 2}\right)$ has full density because $\rS^{-1}(\{1\})=\mathbb{P}^\ast_1$ is sparse.
\item[iii)] The set $\left\{x\in\mathbb{N}:\bP(x)\leqslant\varepsilon x\right\}$ is sparse if $\varepsilon\leqslant 0$ and has full density if $\varepsilon>0$.
\item[iv)] There exists $\theta\in(0,1)$ such that the set $\left\{x\in\mathbb{N}:\bP(x)\leqslant x^\theta\right\}$ has full density.
\end{enumerate}\smallskip
The last two follow from Propositions \ref{prop:epsilon} and \ref{prop:Bertrandrefinement}, respectively.

Every $x\in\mathbb{N}_{\geqslant 2}\backslash\mathbb{P}$ and its image under the \textit{prime map} is coprime; for if they have a prime common divisor then it must divide $\mathbf{p}(x)$, and so must equal to $\mathbf{p}(x)$, implying that $x\geqslant 2\mathbf{p}(x)$ which contradicts Bertrand's Postulate. Under the \textit{prime-power map}, however, the analogue does not hold. Indeed, the number $34\in\mathbb{N}_{\geqslant2}\backslash\mathbb{P}^\ast$ and its image $\bP(34)=2$ are not coprime. Nonetheless, we shall now prove that $x$ and $\bP(x)$ are coprime for almost every $x\in\mathbb{N}$. For this purpose, we define
$$\mathbb{G}:=\left\{x\in\mathbb{N}:\gcd(x,\bP(x))>1\right\},$$
and we choose numbers $X_1\in\mathbb{N}$ and $\theta\in(0,1)$ such that for every $x\geqslant X_1$ the interval $\left[x,x+x^\theta\right]$ contains a prime number\footnote{The existence is again guaranteed by \cite{Baker}.}. Our aim is to show that $\bd\mathbb{G}=0$ via the following lemma.\bigskip

\begin{lemma}\label{lemma:dG}
For every prime-power $p^k\geqslant X_1$, where $p\in\mathbb{P}$ and $k\in\mathbb{N}$, we have
$$\textstyle\left|\left[\mathbb{G}\cap \bq^{-1}\left(\left\{p^k\right\}\right)\right]\backslash\mathbb{P}^\ast\right|\leqslant p^{k\theta-1}.$$
\end{lemma}
\begin{proof}
\sloppy Let $p^k\geqslant X_1$ be a prime-power, where $p\in\mathbb{P}$ and $k\in\mathbb{N}$. For every $x\in \left[\mathbb{G}\cap \bq^{-1}\left(\left\{p^k\right\}\right)\right]\backslash\mathbb{P}^\ast$ we have that $x\notin\mathbb{P}^\ast$ is a preimage of $p^k$ under $\bq$ satisfying
$$\textstyle 1<\gcd(x,\bP(x))=\gcd\left(x,x-p^k\right)=\gcd\left(x,p^k\right),$$
so $x$ is a multiple of $p$, and that $x\in\left(p^k,p^k+p^{k\theta}\right]$. Therefore, the number of elements of $\left[\mathbb{G}\cap \bq^{-1}\left(\left\{p^k\right\}\right)\right]\backslash\mathbb{P}^\ast$ is bounded above by the number of multiples of $p$ in the interval $\left(p^k,p^k+p^{k\theta}\right]$, namely
$$\left|\left\{p^k+p,\ldots,p\left\lfloor\frac{p^k+p^{k\theta}}{p}\right\rfloor\right\}\right|=\frac{p\left\lfloor\frac{p^k+p^{k\theta}}{p}\right\rfloor-p^k}{p}\leqslant p^{k\theta-1},$$
proving the lemma.
\end{proof}\bigskip

\begin{theorem}\label{thm:coprime}
The set $\mathbb{G}$ is sparse.
\end{theorem}
\begin{proof}
Since
$$\bd\mathbb{G}=\bd\left(\mathbb{G}\cap\mathbb{P}^\ast\right)+\bd\left(\mathbb{G}\backslash\mathbb{P}^\ast\right)$$
and $\bd\left(\mathbb{G}\cap\mathbb{P}^\ast\right)=0$, it suffices to prove that $\bd\left(\mathbb{G}\backslash\mathbb{P}^\ast\right)=0$.

Notice that for every $x\in\mathbb{N}$ we have
$$\bd_x\left(\mathbb{G}\backslash\mathbb{P}^\ast\right)=\frac{\left|\left(\mathbb{G}\backslash\mathbb{P}^\ast\right)\cap\{1,\ldots,x\}\right|}{x}\leqslant \frac{1}{x}\sum_{1\leqslant p^k\leqslant x}\textstyle\left|\left[\mathbb{G}\cap \bq^{-1}\left(\left\{p^k\right\}\right)\right]\backslash\mathbb{P}^\ast\right|,$$
where the summation is taken over all prime-powers $p^k\in\{1,\ldots,x\}$. Now, the last expression is equal to
$$\frac{1}{x}\sum_{1\leqslant p\leqslant x}\sum_{k=1}^{\left\lfloor\log_p x\right\rfloor}\textstyle\left|\left[\mathbb{G}\cap \bq^{-1}\left(\left\{p^k\right\}\right)\right]\backslash\mathbb{P}^\ast\right|,$$
where the outer summation is taken over all primes $p\in\{1,\ldots,x\}$. By Lemma \ref{lemma:dG}, this expression is asymptotically bounded above by
$$\frac{1}{x}\sum_{X_1\leqslant p\leqslant x}\sum_{k=1}^{\left\lfloor\log_p x\right\rfloor}p^{k\theta-1},$$
where the outer summation is taken over all primes $p\in\left\{X_1,\ldots,x\right\}$. This expression is bounded above by
$$\frac{1}{x}\sum_{j=2}^x\sum_{k=1}^{\left\lfloor\log_j x\right\rfloor}j^{k\theta-1} = \frac{1}{x}\sum_{j=2}^x \frac{j^{\theta-1}\left[\left(j^\theta\right)^{\left\lfloor\log_j x\right\rfloor}-1\right]}{j^\theta-1}\sim\frac{1}{x^{1-\theta}}\sum_{j=2}^x\frac{1}{j}\leqslant \frac{\ln x}{x^{1-\theta}}.$$
Since $\theta\in(0,1)$, the last expression vanishes as $x\to\infty$, so we have proved the theorem.
\end{proof}\bigskip

\subsection{Connection to prime-power representations} As the \textit{prime map} is closely connected to \textbf{prime representations} \cite[page 159]{Pillai}, so is the \textit{prime-power map} to \textbf{prime-power representations} of positive integers. More precisely, every initial condition $x\in\mathbb{N}$ is associated to, besides its orbit, the sequence $\left(\bq\left(x_n\right)\right)_{n=1}^{\rS(x)}$ which gives a way to express $x$ as the sum of elements of $\mathbb{P}^\ast_1$:
\begin{eqnarray}\label{eq:primepowerrepresentation}
x &=& x_1\nonumber\\
  &=& \bq\left(x_1\right) + x_2\nonumber\\
  &=& \bq\left(x_1\right) + \bq\left(x_2\right) + x_3\nonumber\\
  &=& \cdots\nonumber\\
  &=& \bq\left(x_1\right) + \bq\left(x_2\right)+\cdots+\bq\left(x_{\rS(x)-1}\right)+x_{\rS(x)}\nonumber\\
  &=& \bq\left(x_1\right) + \bq\left(x_2\right)+\cdots+\bq\left(x_{\rS(x)-1}\right)+\bq\left(x_{\rS(x)}\right).
\end{eqnarray}
Here we have that
$$\bq\left(x_1\right)> \bq\left(x_2\right)> \cdots > \bq\left(x_{\rS(x)}\right)$$
which follows from the following two lemmas.\bigskip

\begin{lemma}\label{lemma:monotonicity}
For every $x,y\in\mathbb{N}$, if $x\leqslant y$ then $\bq(x)\leqslant \bq(y)$.
\end{lemma}
\begin{proof}
If the hypothesis holds and $\bq(x)>\bq(y)$, then $\bq(y)<\bq(x)\leqslant x\leqslant y$, so $\bq(x)$ is an element of $\mathbb{P}^\ast_1$ not exceeding $y$ which is greater than $\bq(y)$, contradicting the definition of $\bq(y)$.
\end{proof}\bigskip


\begin{lemma}\label{lemma:monotonicity2}
For every initial condition $x\in\mathbb{N}$, the sequence $\left(\bq\left(x_n\right)\right)_{n=1}^{\rS(x)}$ is monotonically decreasing.
\end{lemma}
\begin{proof}
Let $x\in\mathbb{N}$. If $\rS(x)=1$ then there is nothing to prove. Otherwise, we have $\bq\left(x_1\right)\geqslant \bq\left(x_2\right)\geqslant \cdots \geqslant \bq\left(x_{\rS(x)}\right)$ by Lemma \ref{lemma:monotonicity}. Suppose for a contradiction that there exists $n\in\{1,\ldots,\rS(x)-1\}$ such that $\bq\left(x_n\right)=\bq\left(x_{n+1}\right)$. Then $$2\bq\left(x_n\right)=\bq\left(x_n\right)+\bq\left(x_{n+1}\right)=\left(x_n-x_{n+1}\right)+\left(x_{n+1}-x_{n+2}\right)\leqslant x_n.$$
But between $\bq\left(x_n\right)$ and $2\bq\left(x_n\right)$ there is a prime by Bertrand's Postulate, contradicting the definition of $\bq\left(x_n\right)$.
\end{proof}\bigskip

Clearly, there could be many ways to represent $x\in\mathbb{N}$ as the sum of elements of $\mathbb{P}^\ast_1$. The one given by \eqref{eq:primepowerrepresentation} is the result of applying the \textbf{greedy algorithm}: for each $n\in\{1,\ldots,\rS(x)\}$, one chooses the largest element $\bq\left(x_n\right)$ of $\mathbb{P}^\ast_{1}$ not exceeding $x_n$. As in \cite[page 164]{Pillai} and \cite[page 695]{Luca}, a question one could ask is, for every $s\in\mathbb{N}$, what is the smallest positive integer which is expressed by the greedy algorithm as the sum of exactly $s$ elements of $\mathbb{P}^\ast_1$? These are precisely the numbers $\xi_s:=\min\rS^{-1}(\{s\})$ for every $s\in\mathbb{N}$. The values of $\xi_1$, $\xi_2$, $\xi_3$, and $\xi_4$ are
$$1=1,\quad 6=5+1,\quad 95=89+5+1,\quad\text{and}\quad 360748=360653+89+5+1,$$
respectively. Anyone attempting to compute the subsequent values should take into consideration the following facts.\bigskip

\begin{proposition}\label{prop:xis}\textcolor{white}{a}
\begin{enumerate}
\item[i)] For every $s\in\mathbb{N}$ we have $\bP\left(\xi_{s+1}\right)=\xi_s$.
\item[ii)] For every $s\in\mathbb{N}$ we have $\xi_{s+1}=\min\bP^{-1}\left(\left\{\xi_{s}\right\}\right)$.
\item[iii)] For every integer $s\geqslant 3$, the number $\xi_s$ is not a prime-power.
\item[iv)] We have $$\frac{\xi_{s+1}}{\xi_s},\frac{\bq\left(\xi_{s+1}\right)}{\bq\left(\xi_s\right)}\xrightarrow{x\to\infty}\infty.$$
\end{enumerate}
\end{proposition}
\begin{proof}
First, we prove i). Let $s\in\mathbb{N}$. Since $\rS\left(\xi_{s+1}\right)=s+1$, then $\rS\left(\bP\left(\xi_{s+1}\right)\right)=s$, so $\xi_s\leqslant \bP\left(\xi_{s+1}\right)$. To prove that $\xi_s=\bP\left(\xi_{s+1}\right)$, suppose for a contradiction that $\xi_s< \bP\left(\xi_{s+1}\right)$. Then,
$$\bq\left(\xi_{s+1}\right)<\bq\left(\xi_{s+1}\right)+\xi_s<\bq\left(\xi_{s+1}\right)+\bP\left(\xi_{s+1}\right)=\xi_{s+1}.$$
Since $\bq\left(\xi_{s+1}\right)$ is the largest prime-power not exceeding $\xi_{s+1}$, this implies that $\bq\left(\xi_{s+1}\right)+\xi_s$ is not a prime-power, and that $\bq\left(\xi_{s+1}\right)$ is also the largest prime-power not exceeding $\bq\left(\xi_{s+1}\right)+\xi_s$. The latter implies that
$$\bP\left(\bq\left(\xi_{s+1}\right)+\xi_s\right)=\left[\bq\left(\xi_{s+1}\right)+\xi_s\right]-\bq\left(\xi_{s+1}\right)=\xi_s,$$
so $\bq\left(\xi_{s+1}\right)+\xi_s$ is a positive integer less than $\xi_{s+1}$ having settling time
$$\rS\left(\bq\left(\xi_{s+1}\right)+\xi_s\right)=\rS\left(\xi_s\right)+1=s+1,$$
contradicting the definition of $\xi_{s+1}$.

Both ii) and iii) are immediate consequences of i). To prove iv), notice that, by Proposition \ref{prop:Bertrandrefinement}, there exists $\theta\in(0,1)$ such that
$$\xi_{s}=\bP\left(\xi_{s+1}\right)\leqslant{\xi_{s+1}}^\theta,\qquad\text{i.e.},\qquad \frac{\xi_{s+1}}{\xi_s}\geqslant {\xi_{s+1}}^{1-\theta},$$
for all sufficiently large values of $s$. This implies
$$\frac{\xi_{s+1}}{\xi_s}\xrightarrow{s\to\infty}\infty\qquad\text{and}\qquad\frac{\bq\left(\xi_{s+1}\right)}{\bq\left(\xi_s\right)}=\frac{\xi_{s+1}-\xi_s}{\xi_s-\xi_{s-1}}=\frac{\frac{\xi_{s+1}}{\xi_s}-1}{1-\frac{1}{\left(\frac{\xi_s}{\xi_{s-1}}\right)}}\xrightarrow{s\to\infty}\infty,$$
as desired.
\end{proof}\bigskip

Part i) of Proposition \ref{prop:xis} also means that for every $s\in\mathbb{N}$ we have $\xi_{s+1}=\xi_s+q_1$, where $\left(q_1,q_2\right)$ is the first pair of consecutive prime-powers with $q_2-q_1\geqslant\xi_s+1$. Therefore, to compute $\xi_5$ one must find the first pair of consecutive prime-powers differing by at least $360749$.\bigskip

\section{The family of maps $\bP_p$}\label{sec:Pp}

The unpredictability of the functions $\rT$ and $\rS$ largely originates from the fact that the prime number $p$ acting at each iteration of $\bP$ is not fixed. Take as an initial condition, for instance, $x_1=35$. At the first iteration, which produces $x_2=\bP(35)=35-2^5=3$, the acting prime is $2$. However, at the second, which produces $x_3=\bP(3)=\frac{3}{3}=1$, the acting prime is $3$. In order to eliminate this unpredictability, and for a comparison, let us now consider a modification of the map $\bP$ constructed by fixing the acting prime. This consists in a one-parameter family of maps $\bP_p$, where $p\in\mathbb{P}$. As we shall see, all orbits of $\bP_p$ are predictable via the base-$p$ representations of the iterates (Lemma \ref{transmap}), and hence so are their transient lengths and settling times (Theorem \ref{thm:Fr}).\bigskip

\subsection{The maps $\bP_p$}
 
Fix a prime $p\in\mathbb{P}$. Define the map $\bP_p:\mathbb{N}\to\mathbb{N}$ by
\begin{equation}\label{eq:Fr}
\bP_p(x):=\left\{\begin{array}{cl}
x&\text{if }1\leqslant x\leqslant p-1,\\
x-\bq_p(x)&\text{if }x\text{ is not a power of }p,\\
\frac{x}{p}&\text{if }x\text{ is a power of }p,
\end{array}\right.
\end{equation}
where $\bq_p(x)$ is the largest power of $p$ not exceeding $x$. Notice that this map has $p-1$ distinct fixed points: $1$, \ldots, $p-1$.

As in the original map, the \textbf{orbit} of $x\in\mathbb{N}$ under the map $\bP_p:\mathbb{N}\to\mathbb{N}$ is the sequence $\left(x_n\right)_{n=1}^\infty$ where
$$x_1=x\qquad\quad\text{and}\quad\qquad x_{n+1}=\bP_p\left(x_n\right)\quad\text{for every }n\in\mathbb{N}.$$
The \textbf{transient length} and \textbf{limit} of $x$ under $\bP_p$ are the positive integers
$$\rT_p(x):=\min\bigl\{n\in\mathbb{N}:x_n\in\{1,\ldots,p-1\}\bigr\}\qquad\text{and}\qquad \rL_p(x):=x_{\rT_p(x)},$$
respectively. Equivalently, $\rL_p(x)=\lim\limits_{n\to\infty} x_n$, and $\rT_p(x)$ is the index of the first term in the orbit having value $\rL_p(x)$. One immediately sees that $\rT_p$ is unbounded since it diverges along, e.g., the sequence of powers of $p$. The \textbf{settling time} and the \textbf{attractor} of $x$ under $\bP_p$ are the positive integers
$$\rS_p(x):=\min\bigl\{n\in\mathbb{N}:x_n\in\{1,\ldots,p-1\}\cup\bigl\{p^k:k\in\mathbb{N}\bigr\}\bigr\}\qquad\text{and}\qquad \rA_p(x):=x_{\rS_p(x)},$$
respectively.

Our aim is to derive explicit formulae for $\rT_p(x)$, $\rL_p(x)$, $\rS_p(x)$, and $\rA_p(x)$ in terms of $p$ and $x$. This will be achieved using some number-theoretic tools; see \cite{HardyWright,Koshy} for background. First, the \textbf{$p$-adic value} \cite[page 562]{HardyWright} of $x$ is the non-negative integer
$$\nu_p(x):=\max\left\{m\in\mathbb{N}_0:p^m\mid x\right\},$$
i.e., the exponent of the largest power of $p$ which divides $x$. Moreover, we denote by $x\Mod p$ the remainder when $x$ is divided by $p$ \cite[page 71]{Koshy}, $\cS_p(x)$ the sum of the coefficients in the \textbf{base-$p$ representation} \cite[Section 2.2]{Koshy} of $x$, and
$$\cS_p^\ast(x):=\cS_p(x)-(x\Mod p).$$
The formulae are given by the following theorem.\bigskip

\begin{theorem}\label{thm:Fr}
Let $p\in\mathbb{P}$. For every $x\in\mathbb{N}$ we have
\begin{eqnarray*}
\rT_p(x)&=&
\begin{cases}
\cS_p^\ast(x)+\nu_p(x)&\text{if }p\mid x,\\
\cS_p^\ast(x)+1&\text{otherwise},
\end{cases}\\
\rL_p(x)&=&\begin{cases}
1&\text{if }p\mid x,\\
x\Mod p&\text{otherwise},
\end{cases}\\
\rS_p(x)&=&\begin{cases}
\cS_p^\ast(x)&\text{if }p\mid x,\\
\cS_p^\ast(x)+1&\text{otherwise},
\end{cases}\\
\rA_p(x)&=&\begin{cases}
p^{\nu_p(x)}&\text{if }p\mid x,\\
x\Mod p&\text{otherwise}.
\end{cases}
\end{eqnarray*}
\end{theorem}\bigskip

To prove this theorem, we first study the structure of the orbits of the map $\bP_p$. This will be done by looking at how the base-$p$ representation of $x$ changes under an application of $\bP_p$. More formally, we study the conjugate map $\overline{\bP}_p:=\Psi_p\circ \bP_p\circ {\Psi_p}^{-1}$, where $\Psi_p$ is the bijection associating to every positive integer the infinite sequence whose terms are the coefficients $a_0,\ldots,a_k\in\mathbb{Z}_p$, where $a_k\neq0$ and $k\in\mathbb{N}_0$, in its base-$p$ representation $\sum_{i=0}^k a_ip^i$ followed by infinitely many zeros\footnote{Denoted using the standard notation $\overline{0}$.} [e.g., $\Psi_2(6)=\left(0,1,1,\overline{0}\right)$], as a self-map on the space $\Omega_p$ of all such sequences (see Figure \ref{fig:conjugation}). The action of the map $\overline{\bP}_p$ is described by the following lemma.\bigskip

\begin{figure}
\centering
\input{fig-conjugation}
\caption{\label{fig:conjugation}\small
The commutative diagram representing the conjugation $\overline{\bP}_p=\Psi_p\circ \bP_p\circ {\Psi_p}^{-1}$.}
\end{figure}

\begin{lemma}\label{transmap}
The image of $\left(a_0,\ldots,a_k,\overline{0}\right)\in\Omega_p$, where $a_k\neq0$ and $k\in\mathbb{N}_0$, under $\overline{\bP}_p$ is given by
\begin{equation}\label{eq:transmap}
\overline{\bP}_p\left(a_0,\ldots,a_k,\overline{0}\right)=\begin{cases}
\left(a_0,\ldots,a_k,\overline{0}\right)&\text{if }k=0,\\
\left(a_1,\ldots,a_k,\overline{0}\right)&\text{if }k\geqslant 1\text{ and }a_0=\cdots=a_{k-1}=0\text{ and }a_k=1,\\
\left(a_0,\ldots,a_k-1,\overline{0}\right)&\text{otherwise}. 
\end{cases}
\end{equation}
\end{lemma}
\begin{proof}
Let $x:={\Psi_p}^{-1}\left(a_0,\ldots,a_k,\overline{0}\right)\in\mathbb{N}$. This means that $\sum_{i=0}^k a_ip^i$ is the base-$p$ representation of $x$.
\begin{enumerate}
\item[i)] If $k=0$ then $1\leqslant x\leqslant p-1$, so $\bP_p(x)=x$. In other words, $\overline{\bP}_p\left(a_0,\ldots,a_k,\overline{0}\right)=\left(a_0,\ldots,a_k,\overline{0}\right)$.\medskip
\item[ii)] Suppose $k\geqslant1$, $a_0=a_1=\cdots=a_{k-1}=0$, and $a_k=1$. Then $x=p^k\geqslant p$, so $\bP_p(x)=\frac{x}{p}=p^{k-1}$. In other words, $\overline{\bP}_p\bigl(\underbrace{0,\ldots,0}_k,1,\overline{0}\bigr)=\bigl(\underbrace{0,\ldots,0}_{k-1},1,\overline{0}\bigr)$, i.e., $\overline{\bP}_p\left(a_0,\ldots,a_k,\overline{0}\right)=\left(a_1,\ldots,a_k,\overline{0}\right)$.
\item[iii)] Suppose $k\geqslant1$, and suppose at least one of the following is not true: $a_0=\cdots=a_{k-1}=0$ and $a_k=1$. Then, $x$ is not a power of $p$; for otherwise, since $x>p^k$, we must have that $x=p^\ell$ for some integer $\ell>k$, meaning that $p^\ell$ is the unique base-$p$ representation of $x$, in which the leading coefficient is $1$ and all the other coefficients are $0$, contradicting our supposition. Since $p^k$ is the largest power of $p$ not exceeding $x$, then $\bP_p(x)=x-p^k=\sum_{i=0}^{k-1} a_ip^i + \left(a_k-1\right)p^k$. In other words, $\overline{\bP}_p\left(a_0,\ldots,a_k,\overline{0}\right)=\left(a_0,\ldots,a_k-1,\overline{0}\right)$.
\end{enumerate}
\end{proof}\bigskip

For every $x\in\mathbb{N}$, Lemma \ref{transmap} gives a complete description of the orbit of $\Psi_p(x)$ under $\overline{\bP}_p$, and hence that of the orbit of $x$ under $\bP_p$. For instance, a complete description of the orbit of $108$ under $\bP_2$,
$$108\mapsto 44\mapsto 12\mapsto 4\mapsto 2\mapsto 1,$$
can be obtained from the fact that $\Psi_2(108)=\left(0,0,1,1,0,1,1,\overline{0}\right)$ and, under $\overline{\bP}_2$,
$$\left(0,0,1,1,0,1,1,\overline{0}\right)\mapsto \left(0,0,1,1,0,1,\overline{0}\right)\mapsto \left(0,0,1,1,\overline{0}\right)\mapsto \left(0,0,1,\overline{0}\right)\mapsto \left(0,1,\overline{0}\right)\mapsto \left(1,\overline{0}\right).$$

Now, let $\alpha=\left(a_k,\ldots,a_0\right)$. Based on the definitions of the analogues under $\bP_p$, the \textbf{transient length}, \textbf{limit}, \textbf{settling time}, and \textbf{attractor} of $\alpha$ under $\overline{\bP}_p$ are
\begin{eqnarray*}
\overline{\rT}_p(\alpha)&=&\min\bigl\{n\in\mathbb{N}:\text{only the first term of }\alpha_n\text{ is non-zero}\bigr\},\\
\overline{\rL}_p(\alpha)&=&\alpha_{\overline{\rT}_p(\alpha)},\\
\overline{\rS}_p(\alpha)&=&\min\bigl\{n\in\mathbb{N}:
\alpha_n\in\bigl\{ \bigl( \underbrace{0,0,\ldots,0}_t,1,\overline{0}\bigr):t\in\mathbb{N}_0\bigr\}\cup
\bigl\{\bigl(1,\overline{0}\bigr),\ldots,\bigl(p-1,\overline{0}\bigr)\bigr\}\bigr\},\\
\overline{\rA}_p(\alpha)&=&\alpha_{\overline{\rS}_p(\alpha)},
\end{eqnarray*}
respectively, where $\left(\alpha_n\right)_{n=1}^\infty$ is the orbit of $\alpha$ under $\overline{\bP}_p$, which is determined by
$$\alpha_1=\alpha\qquad\quad\text{and}\quad\qquad\alpha_{n+1}=\overline{\bP}_p\left(\alpha_n\right)\quad\text{for every }n\in\mathbb{N}.$$\smallskip

\paragraph{\textit{Proof of Theorem \ref{thm:Fr}}} Proving Theorem \ref{thm:Fr} is equivalent to proving that
\begin{eqnarray*}
\overline{\rT}_p(\alpha)&=&
\begin{cases}
\nu_p\left({\Psi_p}^{-1}(\alpha)\right)+a_k+\cdots+a_1&\text{if }a_0=0,\\
1+a_k+\cdots+a_1&\text{if }a_0\geqslant 1,
\end{cases}\\
\overline{\rL}_p(\alpha)&=&\begin{cases}
\left(1,\overline{0}\right)&\text{if }a_0=0,\\
\left(a_0,\overline{0}\right)&\text{if }a_0\geqslant 1,
\end{cases}\\
\overline{\rS}_p(\alpha)&=&\begin{cases}
a_k+\cdots+a_1&\text{if }a_0=0,\\
1+a_k+\cdots+a_1&\text{if }a_0\geqslant 1,
\end{cases}\\
\overline{\rA}_p(\alpha)&=&\begin{cases}
\bigl( \underbrace{0,0,\ldots,0}_{\nu_p\left({\Psi_p}^{-1}(\alpha)\right)},1,\overline{0}\bigr)&\text{if }a_0=0,\\
\left(a_0,\overline{0}\right)&\text{if }a_0\geqslant 1.
\end{cases}
\end{eqnarray*}
We consider each of the two cases.\bigskip

\noindent\underline{\textsc{Case I}}: $a_0\geqslant 1$. If $k=0$, then the orbit of $\alpha=\left(a_0,\overline{0}\right)$ is constant, so $\overline{\rL}_p(\alpha)=\overline{\rA}_p(\alpha)=\left(a_0,\overline{0}\right)$ and $\overline{\rT}_p(\alpha)=\overline{\rS}_p(\alpha)=1$, obeying the desired formulae. Now suppose $k\geqslant 1$. Since $a_0\geqslant 1$, then each iteration is prescribed by the last branch in \eqref{eq:transmap}. Since
$$\alpha_1=\left(a_0,\ldots,a_k,\overline{0}\right),$$
then
$$\alpha_{1+a_k}={\overline{\bP}_p}^{a_k}\left(\alpha_1\right)=\left(a_0,\ldots,a_{k-1},\overline{0}\right).$$
Continuing this, we have, for every $j\in\{k-1,\ldots,1\}$,
$$\alpha_{1+a_k+\cdots+a_j}={\overline{\bP}_p}^{a_{j}}\left(\alpha_{1+a_k+\cdots+a_{j+1}}\right)=\left(a_0,\ldots,a_{j-1},\overline{0}\right),$$
which means that
$$\overline{\rT}_p(\alpha)=\overline{\rS}_p(\alpha)=1+a_k+\cdots+a_1\qquad\text{and}\qquad\overline{\rL}_p(\alpha)=\overline{\rA}_p(\alpha)=\left(a_0,\overline{0}\right),$$
as desired.\bigskip

\noindent\underline{\textsc{Case II}}: $a_0=0$. In this case, $k\geqslant 1$. Let $t:=\nu_p\left({\Psi_p}^{-1}(\alpha)\right)$, then $a_0=\cdots=a_{t-1}=0$ and $a_t\geqslant 1$, so
$$\alpha_1=\bigl( \underbrace{0,\ldots,0}_t,a_t,\ldots,a_k,\overline{0}\bigr).$$
By the last branch in \eqref{eq:transmap}, we have
$$\alpha_{1+a_k}={\overline{\bP}_p}^{a_k}\left(\alpha_1\right)=\bigl( \underbrace{0,\ldots,0}_t,a_t,\ldots,a_{k-1},\overline{0}\bigr),$$
and, for every $j\in\{k-1,\ldots,t+1\}$,
$$\alpha_{1+a_k+\cdots+a_{j}}={\overline{\bP}_p}^{a_{j}}\left(\alpha_{1+a_k+\cdots+a_{j+1}}\right)=\bigl( \underbrace{0,\ldots,0}_{t},a_t,\ldots,a_{j-1},\overline{0}\bigr).$$
Next,
$$\alpha_{a_k+\cdots+a_t}={\overline{\bP}_p}^{a_{t}-1}\left(\alpha_{1+a_k+\cdots+a_{t+1}}\right)=\bigl( \underbrace{0,\ldots,0}_t,1,\overline{0}\bigr).$$
At this point we have proved that
$$\overline{\rS}_p(\alpha)=a_k+\cdots+a_1\qquad\text{and}\qquad\overline{\rA}_p(\alpha)=\bigl( \underbrace{0,\ldots,0}_t,1,\overline{0}\bigr)=\bigl( \underbrace{0,0,\ldots,0}_{\nu_p\left({\Psi_p}^{-1}(\alpha)\right)},1,\overline{0}\bigr).$$
Finally, the dynamics concludes with $t$ iterations prescribed by the second branch in \eqref{eq:transmap}:
$$\alpha_{a_k+\cdots+a_t+t}={\overline{\bP}_p}^t\left(\alpha_{a_k+\cdots+a_t}\right)=\left(1,\overline{0}\right).$$
Therefore,
$$\overline{\rT}_p(\alpha)=t+a_k+\cdots+a_1=\nu_p\left({\Psi_k}^{-1}(\alpha)\right)+a_k+\cdots+a_1\qquad\text{and}\qquad\overline{\rL}_p(\alpha)=\left(1,\overline{0}\right),$$
as desired.\hfill$\square$\bigskip

\subsection{A comparison of $\bP_p$ and $\bP$}

Let $p\in\mathbb{P}$. For every $x\in\mathbb{N}$, let
$$\hat{\rT}_p(x):=\frac{1}{x}\sum_{t=1}^x \rT_p(t)\qquad\quad\text{and}\quad\qquad \hat{\rS}_p(x):=\frac{1}{x}\sum_{t=1}^x \rS_p(t)$$
be the $x$-th Ces\'aro mean of the sequences $\left(\rT_p(x)\right)_{x=1}^\infty$ and $\left(\rS_p(x)\right)_{x=1}^\infty$, respectively. From Theorem \ref{thm:Fr}, noticing that
$$\frac{\bigl|\{t\leqslant x:p\mid t\}\bigr|}{x}=\frac{1}{x}\left\lfloor\frac{x}{p}\right\rfloor\sim\frac{1}{p}$$
and
$$\frac{1}{x}\sum_{t=1}^x \nu_p(t)=\frac{1}{x}\sum_{t=1}^{\left\lfloor\log_p x\right\rfloor}\left\lfloor\frac{x}{p^t}\right\rfloor\sim\sum_{t=1}^{\left\lfloor\log_p x\right\rfloor}\frac{1}{p^t}\sim\frac{1}{p-1}$$
and applying \cite[Theorem 1]{Bush},
$$\frac{1}{x}\sum_{t=1}^x\cS_p(t)\sim \frac{(p-1)\ln x}{2\ln p},$$
one proves the following corollary.\bigskip

\begin{corollary}\label{corollary}
We have
$$\hat{\rT}_p(x)\sim\hat{\rS}_p(x)\sim \frac{(p-1)\ln x}{2\ln p}.$$
\end{corollary}\bigskip

\begin{figure}
\centering
\input{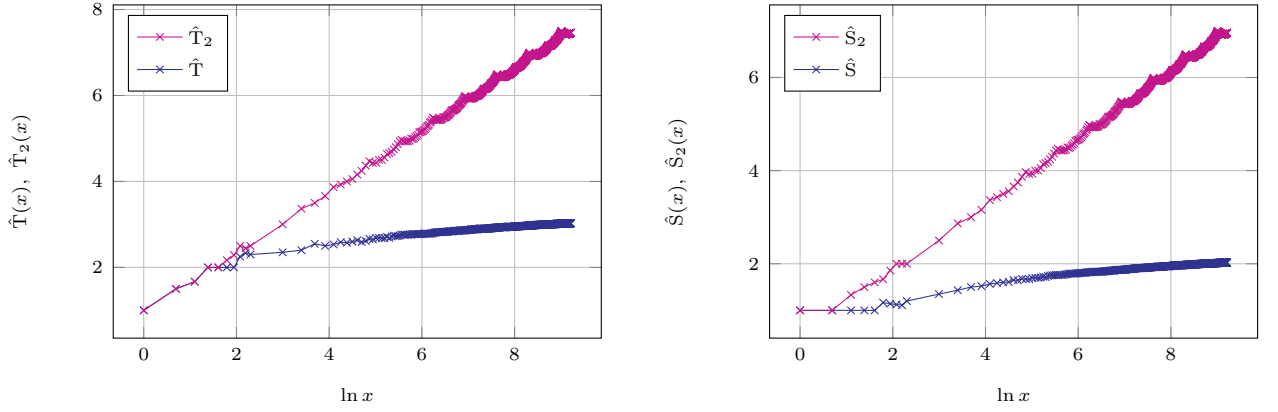}
\caption{\label{fig:TSasymptotics}\small
Plots of $\hat{\rT}(x)$ and $\hat{\rT}_2(x)$ versus $\ln x$ (left) and $\hat{\rS}(x)$ and $\hat{\rS}_2(x)$ versus $\ln x$ (right) for $x\in\left\{1,\ldots,1\cdot 10^4\right\}$.}
\end{figure}

Thus, the growths of $\hat{\rT}_p(x)$ and $\hat{\rS}_p(x)$ with $x$ are both logarithmic. By contrast, letting
$$\hat{\rT}(x):=\frac{1}{x}\sum_{t=1}^x \rT(t)\qquad\quad\text{and}\quad\qquad \hat{\rS}(x):=\frac{1}{x}\sum_{t=1}^x \rS(t),$$
from Theorems \ref{thm:power2} and \ref{thm:orderS(x)} we know that the former grows at most logarithmically, and the latter at most double-logarithmically. In fact, Figure \ref{fig:TSasymptotics} suggests that both quantities grow at most double-logarithmically. This is confirmed by Figure \ref{fig:TSasymptotics2} which suggests the existence of $\lambda\geqslant 1$ such that
\begin{equation}\label{eq:TsimS}
\frac{\hat{\rT}(x)}{\hat{\rS}(x)}\sim\lambda.
\end{equation}

\begin{figure}
\centering
\input{fig-TSasymptotics2}
\caption{\label{fig:TSasymptotics2}\small
Plot of $\hat{\rT}(x)$ versus $\hat{\rS}(x)$ for $x\in\left\{1,\ldots,1\cdot 10^5\right\}$.}
\end{figure}
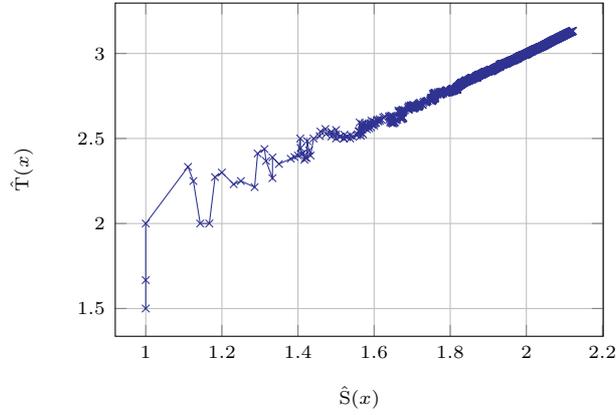

Now, since
\begin{eqnarray*}
\left(\hat{\rT}-\hat{\rS}\right)(x)&=&\frac{1}{x}\sum_{t=1}^x(\rT-\rS)(t)\\
&=&\frac{1}{x}\sum_{u=0}^{\max\limits_{1\leqslant t\leqslant x}(\rT-\rS)(t)}u\,\left|(\rT-\rS)^{-1}(\{u\})\cap\{1,\ldots,x\}\right|\\
&=&\sum_{u=1}^{\max\limits_{1\leqslant t\leqslant x}(\rT-\rS)(t)} u\,\bd_x(\rT-\rS)^{-1}(\{u\}),
\end{eqnarray*}
then we have $$\left(\hat{\rT}-\hat{\rS}\right)(x)\xrightarrow{x\to\infty}0\qquad\text{if}\qquad \bd_x(\rT-\rS)^{-1}(\{u\})\xrightarrow{x\to\infty}0\quad\text{for every }u\in\mathbb{N}.$$
The latter holds if and only if the set $(\rT-\rS)^{-1}\left(\mathbb{N}\right)=\rA^{-1}\left(\mathbb{N}_{\geqslant 2}\right)$ is sparse. Indeed, an extensive set of values of $\bd_x\rA^{-1}\left(\mathbb{N}_{\geqslant 2}\right)$ suggests that this quantity decays algebraically to zero as $x\to\infty$ (Figure \ref{fig:melaluiprimepower}), leading to the following conjecture.\bigskip

\begin{figure}
\centering
\input{fig-melaluiprimepower}
\caption{\label{fig:melaluiprimepower}\small
Plot of $\bd_x\rA^{-1}\left(\mathbb{N}_{\geqslant 2}\right)$ versus $x$ for $x\in\{5\cdot10^4,\ldots,1\cdot 10^6\}$ in log-log scale (left) together with the least-squares line determined using its values for $x\in\{6\cdot 10^5,\ldots,1\cdot 10^6\}$ which has slope $\alpha\approx -0.006$ and ordinate intercept $\beta\approx -0.243$, and in linear-linear scale (right) together with the curve $e^{\beta}x^\alpha$.}
\end{figure}
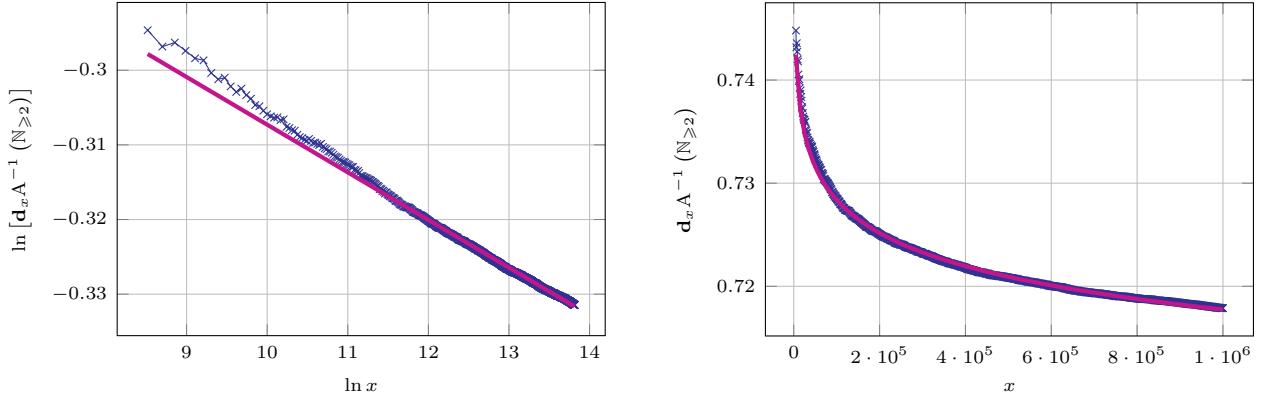

\begin{samepage}
\begin{conjecture}\label{finalconjecture}
There exists a function $\Phi(x)=cx^\alpha$, where $c\in(0,1)$ and $\alpha\in\left(-10^{-2},-10^{-3}\right)$, such that
$$\left|\bd_x\rA^{-1}\left(\mathbb{N}_{\geqslant2}\right) - \Phi(x)\right|=o\left(\Phi(x)\right).$$
In particular, almost every orbit of $\bP$ contains no prime-power.
\end{conjecture}\bigskip
\end{samepage}

\noindent Notice that if this conjecture holds, then so does \eqref{eq:TsimS} with $\lambda=1$.\bigskip



\end{document}

%% file: fig-plotP.tex
\begin{tikzpicture}
\begin{axis}[
	xmin=-17.85714286,
	xmax=267.8571429,
	ymin=0,
	ymax=88.36363636,
	samples=100,
	xlabel=$x$,
	ylabel=$\bP(x)$,
	width=8cm,
    height=6cm,
    clip=false,
    grid=major,clip marker paths=true
]
\addplot[color=newblue,mark=x] plot coordinates {(1, 1) (2, 1) (3, 1) (4, 2) (5, 1) (6, 1) (7, 1) (8, 4) (9, 3) (10, 1) (11, 1) (12, 1) (13, 1) (14, 1) (15, 2) (16, 8) (17, 1) (18, 1) (19, 1) (20, 1) (21, 2) (22, 3) (23, 1) (24, 1) (25, 5) (26, 1) (27, 9) (28, 1) (29, 1) (30, 1) (31, 1) (32, 16) (33, 1) (34, 2) (35, 3) (36, 4) (37, 1) (38, 1) (39, 2) (40, 3) (41, 1) (42, 1) (43, 1) (44, 1) (45, 2) (46, 3) (47, 1) (48, 1) (49, 7) (50, 1) (51, 2) (52, 3) (53, 1) (54, 1) (55, 2) (56, 3) (57, 4) (58, 5) (59, 1) (60, 1) (61, 1) (62, 1) (63, 2) (64, 32) (65, 1) (66, 2) (67, 1) (68, 1) (69, 2) (70, 3) (71, 1) (72, 1) (73, 1) (74, 1) (75, 2) (76, 3) (77, 4) (78, 5) (79, 1) (80, 1) (81, 27) (82, 1) (83, 1) (84, 1) (85, 2) (86, 3) (87, 4) (88, 5) (89, 1) (90, 1) (91, 2) (92, 3) (93, 4) (94, 5) (95, 6) (96, 7) (97, 1) (98, 1) (99, 2) (100, 3) (101, 1) (102, 1) (103, 1) (104, 1) (105, 2) (106, 3) (107, 1) (108, 1) (109, 1) (110, 1) (111, 2) (112, 3) (113, 1) (114, 1) (115, 2) (116, 3) (117, 4) (118, 5) (119, 6) (120, 7) (121, 11) (122, 1) (123, 2) (124, 3) (125, 25) (126, 1) (127, 1) (128, 64) (129, 1) (130, 2) (131, 1) (132, 1) (133, 2) (134, 3) (135, 4) (136, 5) (137, 1) (138, 1) (139, 1) (140, 1) (141, 2) (142, 3) (143, 4) (144, 5) (145, 6) (146, 7) (147, 8) (148, 9) (149, 1) (150, 1) (151, 1) (152, 1) (153, 2) (154, 3) (155, 4) (156, 5) (157, 1) (158, 1) (159, 2) (160, 3) (161, 4) (162, 5) (163, 1) (164, 1) (165, 2) (166, 3) (167, 1) (168, 1) (169, 13) (170, 1) (171, 2) (172, 3) (173, 1) (174, 1) (175, 2) (176, 3) (177, 4) (178, 5) (179, 1) (180, 1) (181, 1) (182, 1) (183, 2) (184, 3) (185, 4) (186, 5) (187, 6) (188, 7) (189, 8) (190, 9) (191, 1) (192, 1) (193, 1) (194, 1) (195, 2) (196, 3) (197, 1) (198, 1) (199, 1) (200, 1) (201, 2) (202, 3) (203, 4) (204, 5) (205, 6) (206, 7) (207, 8) (208, 9) (209, 10) (210, 11) (211, 1) (212, 1) (213, 2) (214, 3) (215, 4) (216, 5) (217, 6) (218, 7) (219, 8) (220, 9) (221, 10) (222, 11) (223, 1) (224, 1) (225, 2) (226, 3) (227, 1) (228, 1) (229, 1) (230, 1) (231, 2) (232, 3) (233, 1) (234, 1) (235, 2) (236, 3) (237, 4) (238, 5) (239, 1) (240, 1) (241, 1) (242, 1) (243, 81) (244, 1) (245, 2) (246, 3) (247, 4) (248, 5) (249, 6) (250, 7)};
\end{axis}
\end{tikzpicture}

%% file: fig-transittime.tex
\begin{tikzpicture}
\begin{axis}[
	xmin=-17.85714286,
	xmax=267.8571429,
	ymin=0,
	ymax=9.780855583,
	samples=100,
	xlabel=$x$,
	ylabel=$\rT(x)$,
	width=8cm,
    height=6cm,
    clip=false,
    grid=major,clip marker paths=true
]
\addplot[color=newblue,mark=x] plot coordinates {(1, 1) (2, 2) (3, 2) (4, 3) (5, 2) (6, 2) (7, 2) (8, 4) (9, 3) (10, 2) (11, 2) (12, 2) (13, 2) (14, 2) (15, 3) (16, 5) (17, 2) (18, 2) (19, 2) (20, 2) (21, 3) (22, 3) (23, 2) (24, 2) (25, 3) (26, 2) (27, 4) (28, 2) (29, 2) (30, 2) (31, 2) (32, 6) (33, 2) (34, 3) (35, 3) (36, 4) (37, 2) (38, 2) (39, 3) (40, 3) (41, 2) (42, 2) (43, 2) (44, 2) (45, 3) (46, 3) (47, 2) (48, 2) (49, 3) (50, 2) (51, 3) (52, 3) (53, 2) (54, 2) (55, 3) (56, 3) (57, 4) (58, 3) (59, 2) (60, 2) (61, 2) (62, 2) (63, 3) (64, 7) (65, 2) (66, 3) (67, 2) (68, 2) (69, 3) (70, 3) (71, 2) (72, 2) (73, 2) (74, 2) (75, 3) (76, 3) (77, 4) (78, 3) (79, 2) (80, 2) (81, 5) (82, 2) (83, 2) (84, 2) (85, 3) (86, 3) (87, 4) (88, 3) (89, 2) (90, 2) (91, 3) (92, 3) (93, 4) (94, 3) (95, 3) (96, 3) (97, 2) (98, 2) (99, 3) (100, 3) (101, 2) (102, 2) (103, 2) (104, 2) (105, 3) (106, 3) (107, 2) (108, 2) (109, 2) (110, 2) (111, 3) (112, 3) (113, 2) (114, 2) (115, 3) (116, 3) (117, 4) (118, 3) (119, 3) (120, 3) (121, 3) (122, 2) (123, 3) (124, 3) (125, 4) (126, 2) (127, 2) (128, 8) (129, 2) (130, 3) (131, 2) (132, 2) (133, 3) (134, 3) (135, 4) (136, 3) (137, 2) (138, 2) (139, 2) (140, 2) (141, 3) (142, 3) (143, 4) (144, 3) (145, 3) (146, 3) (147, 5) (148, 4) (149, 2) (150, 2) (151, 2) (152, 2) (153, 3) (154, 3) (155, 4) (156, 3) (157, 2) (158, 2) (159, 3) (160, 3) (161, 4) (162, 3) (163, 2) (164, 2) (165, 3) (166, 3) (167, 2) (168, 2) (169, 3) (170, 2) (171, 3) (172, 3) (173, 2) (174, 2) (175, 3) (176, 3) (177, 4) (178, 3) (179, 2) (180, 2) (181, 2) (182, 2) (183, 3) (184, 3) (185, 4) (186, 3) (187, 3) (188, 3) (189, 5) (190, 4) (191, 2) (192, 2) (193, 2) (194, 2) (195, 3) (196, 3) (197, 2) (198, 2) (199, 2) (200, 2) (201, 3) (202, 3) (203, 4) (204, 3) (205, 3) (206, 3) (207, 5) (208, 4) (209, 3) (210, 3) (211, 2) (212, 2) (213, 3) (214, 3) (215, 4) (216, 3) (217, 3) (218, 3) (219, 5) (220, 4) (221, 3) (222, 3) (223, 2) (224, 2) (225, 3) (226, 3) (227, 2) (228, 2) (229, 2) (230, 2) (231, 3) (232, 3) (233, 2) (234, 2) (235, 3) (236, 3) (237, 4) (238, 3) (239, 2) (240, 2) (241, 2) (242, 2) (243, 6) (244, 2) (245, 3) (246, 3) (247, 4) (248, 3) (249, 3) (250, 3)};
\addplot[domain=1:250,very thick,newpurple] {ln(x)/ln(2)+1};
\end{axis}
\end{tikzpicture}

%% file: fig-conjugation.tex
\begin{tikzpicture}[scale=1.25]
\node at (-1,1) {$\mathbb{N}$};
\node at (1,1) {$\mathbb{N}$};
\node at (-1,-1) {$\Omega_p$};
\node at (1,-1) {$\Omega_p$};

\draw[->] (-0.6,1) -- (0.6,1);
\node[above] at (0,1) {$\bP_p$};
\draw[->] (-0.6,-1) -- (0.6,-1);
\node[below] at (0,-1) {$\overline{\bP}_p$};

\draw[->] (1,0.6) -- (1,-0.6);
\node[left] at (-1,0) {${\Psi_p}^{-1}$};
\draw[->] (-1,-0.6) -- (-1,0.6);
\node[right] at (1,0) {$\Psi_p$};
\end{tikzpicture}

%% file: fig-TSasymptotics2.tex
\begin{tikzpicture}
\begin{axis}[
	xmin=.9198571430,
	xmax=2.202142857,
	ymin=1.336700000,
	ymax=3.296300000,
	xlabel=$\hat{\rS}(x)$,
	ylabel={$\hat{\rT}(x)$},
	width=8cm,
    height=6cm,
    clip=false,grid=major,legend pos=north west,reverse legend,legend cell align={left}
]
\addplot[color=newblue,mark=x] plot coordinates {(1., 1.500) (1., 1.667) (1., 2.) (1.111, 2.333) (1.125, 2.250) (1.143, 2.) (1.167, 2.) (1.182, 2.273) (1.200, 2.300) (1.231, 2.231) (1.250, 2.250) (1.286, 2.214) (1.294, 2.412) (1.312, 2.438) (1.316, 2.368) (1.333, 2.267) (1.333, 2.389) (1.350, 2.350) (1.381, 2.381) (1.391, 2.391) (1.400, 2.400) (1.406, 2.500) (1.407, 2.444) (1.409, 2.409) (1.414, 2.414) (1.417, 2.375) (1.419, 2.387) (1.423, 2.385) (1.424, 2.485) (1.429, 2.429) (1.433, 2.400) (1.441, 2.500) (1.457, 2.514) (1.459, 2.541) (1.472, 2.556) (1.474, 2.526) (1.487, 2.538) (1.488, 2.512) (1.488, 2.537) (1.500, 2.500) (1.500, 2.524) (1.500, 2.550) (1.510, 2.510) (1.511, 2.511) (1.520, 2.500) (1.521, 2.500) (1.522, 2.522) (1.528, 2.509) (1.529, 2.510) (1.537, 2.500) (1.538, 2.519) (1.545, 2.509) (1.554, 2.518) (1.557, 2.525) (1.559, 2.542) (1.561, 2.544) (1.562, 2.594) (1.565, 2.516) (1.567, 2.533) (1.567, 2.582) (1.569, 2.552) (1.569, 2.585) (1.571, 2.524) (1.574, 2.574) (1.575, 2.562) (1.576, 2.591) (1.577, 2.577) (1.580, 2.580) (1.581, 2.554) (1.583, 2.569) (1.586, 2.586) (1.587, 2.560) (1.590, 2.590) (1.592, 2.566) (1.593, 2.605) (1.595, 2.582) (1.595, 2.583) (1.597, 2.584) (1.598, 2.598) (1.600, 2.575) (1.600, 2.588) (1.603, 2.590) (1.605, 2.593) (1.607, 2.607) (1.609, 2.609) (1.611, 2.600) (1.614, 2.614) (1.615, 2.604) (1.620, 2.609) (1.624, 2.624) (1.628, 2.628) (1.639, 2.629) (1.641, 2.612) (1.642, 2.596) (1.642, 2.632) (1.643, 2.622) (1.644, 2.606) (1.644, 2.624) (1.645, 2.591) (1.645, 2.607) (1.646, 2.593) (1.646, 2.626) (1.646, 2.635) (1.647, 2.618) (1.648, 2.602) (1.648, 2.610) (1.649, 2.588) (1.649, 2.595) (1.650, 2.630) (1.651, 2.613) (1.652, 2.591) (1.652, 2.598) (1.655, 2.595) (1.658, 2.607) (1.661, 2.610) (1.664, 2.656) (1.664, 2.664) (1.667, 2.652) (1.667, 2.659) (1.669, 2.620) (1.669, 2.622) (1.669, 2.654) (1.669, 2.655) (1.669, 2.662) (1.671, 2.650) (1.672, 2.613) (1.672, 2.615) (1.672, 2.632) (1.672, 2.657) (1.672, 2.664) (1.674, 2.652) (1.674, 2.659) (1.674, 2.667) (1.675, 2.617) (1.675, 2.618) (1.675, 2.627) (1.676, 2.655) (1.676, 2.669) (1.677, 2.621) (1.678, 2.664) (1.681, 2.667) (1.689, 2.682) (1.690, 2.669) (1.691, 2.678) (1.691, 2.691) (1.692, 2.671) (1.693, 2.680) (1.693, 2.687) (1.694, 2.687) (1.694, 2.688) (1.695, 2.682) (1.696, 2.684) (1.696, 2.696) (1.697, 2.690) (1.698, 2.686) (1.699, 2.682) (1.699, 2.692) (1.699, 2.693) (1.700, 2.682) (1.700, 2.688) (1.701, 2.678) (1.701, 2.689) (1.702, 2.680) (1.702, 2.684) (1.702, 2.685) (1.702, 2.696) (1.703, 2.676) (1.703, 2.680) (1.703, 2.686) (1.703, 2.691) (1.704, 2.687) (1.704, 2.698) (1.705, 2.678) (1.705, 2.682) (1.705, 2.693) (1.706, 2.683) (1.706, 2.689) (1.707, 2.679) (1.708, 2.686) (1.708, 2.691) (1.710, 2.688) (1.714, 2.688) (1.715, 2.685) (1.715, 2.699) (1.716, 2.687) (1.716, 2.695) (1.716, 2.696) (1.717, 2.690) (1.717, 2.692) (1.717, 2.707) (1.718, 2.688) (1.718, 2.691) (1.718, 2.697) (1.719, 2.695) (1.719, 2.699) (1.719, 2.703) (1.720, 2.704) (1.721, 2.696) (1.721, 2.711) (1.727, 2.698) (1.728, 2.699) (1.729, 2.710) (1.731, 2.716) (1.735, 2.716) (1.736, 2.712) (1.737, 2.714) (1.737, 2.718) (1.738, 2.715) (1.738, 2.719) (1.740, 2.721) (1.741, 2.722) (1.747, 2.724) (1.748, 2.725) (1.749, 2.733) (1.749, 2.735) (1.750, 2.730) (1.750, 2.741) (1.751, 2.722) (1.751, 2.725) (1.751, 2.729) (1.751, 2.731) (1.752, 2.719) (1.752, 2.722) (1.752, 2.726) (1.752, 2.732) (1.753, 2.723) (1.753, 2.727) (1.753, 2.728) (1.753, 2.736) (1.753, 2.737) (1.753, 2.740) (1.754, 2.725) (1.754, 2.728) (1.754, 2.732) (1.754, 2.737) (1.754, 2.738) (1.755, 2.730) (1.755, 2.763) (1.756, 2.731) (1.756, 2.738) (1.756, 2.742) (1.756, 2.753) (1.756, 2.760) (1.757, 2.737) (1.757, 2.739) (1.757, 2.743) (1.757, 2.750) (1.757, 2.757) (1.757, 2.761) (1.757, 2.764) (1.758, 2.734) (1.758, 2.751) (1.758, 2.754) (1.758, 2.758) (1.758, 2.761) (1.758, 2.762) (1.758, 2.766) (1.759, 2.735) (1.759, 2.739) (1.759, 2.752) (1.759, 2.753) (1.759, 2.755) (1.759, 2.756) (1.759, 2.759) (1.759, 2.763) (1.759, 2.766) (1.760, 2.736) (1.760, 2.740) (1.760, 2.753) (1.760, 2.756) (1.760, 2.757) (1.760, 2.760) (1.760, 2.764) (1.760, 2.767) (1.760, 2.768) (1.761, 2.741) (1.761, 2.758) (1.761, 2.761) (1.761, 2.764) (1.761, 2.769) (1.762, 2.758) (1.766, 2.759) (1.767, 2.760) (1.767, 2.767) (1.768, 2.772) (1.772, 2.772) (1.773, 2.760) (1.773, 2.764) (1.773, 2.773) (1.774, 2.758) (1.774, 2.759) (1.774, 2.761) (1.775, 2.759) (1.775, 2.762) (1.775, 2.763) (1.775, 2.768) (1.775, 2.772) (1.776, 2.763) (1.776, 2.764) (1.776, 2.766) (1.776, 2.769) (1.777, 2.770) (1.777, 2.771) (1.777, 2.774) (1.778, 2.775) (1.780, 2.765) (1.781, 2.765) (1.782, 2.772) (1.782, 2.776) (1.786, 2.777) (1.787, 2.771) (1.787, 2.773) (1.787, 2.774) (1.787, 2.777) (1.788, 2.771) (1.788, 2.772) (1.788, 2.773) (1.788, 2.774) (1.788, 2.775) (1.788, 2.777) (1.789, 2.771) (1.789, 2.772) (1.789, 2.773) (1.789, 2.774) (1.789, 2.775) (1.789, 2.776) (1.789, 2.777) (1.790, 2.771) (1.790, 2.772) (1.790, 2.774) (1.790, 2.775) (1.790, 2.776) (1.790, 2.777) (1.790, 2.778) (1.790, 2.779) (1.790, 2.781) (1.790, 2.784) (1.791, 2.773) (1.791, 2.775) (1.791, 2.776) (1.791, 2.777) (1.791, 2.778) (1.791, 2.779) (1.791, 2.782) (1.791, 2.783) (1.792, 2.777) (1.792, 2.780) (1.792, 2.781) (1.792, 2.783) (1.793, 2.773) (1.793, 2.776) (1.794, 2.771) (1.794, 2.772) (1.794, 2.774) (1.795, 2.772) (1.795, 2.775) (1.795, 2.777) (1.795, 2.778) (1.796, 2.776) (1.797, 2.775) (1.798, 2.773) (1.798, 2.774) (1.799, 2.774) (1.799, 2.776) (1.799, 2.777) (1.800, 2.778) (1.800, 2.781) (1.801, 2.780) (1.802, 2.778) (1.802, 2.781) (1.802, 2.783) (1.802, 2.784) (1.802, 2.785) (1.803, 2.779) (1.803, 2.784) (1.804, 2.784) (1.804, 2.786) (1.804, 2.787) (1.804, 2.788) (1.805, 2.779) (1.805, 2.785) (1.805, 2.786) (1.805, 2.787) (1.805, 2.789) (1.805, 2.791) (1.806, 2.778) (1.806, 2.779) (1.806, 2.780) (1.806, 2.781) (1.806, 2.782) (1.806, 2.785) (1.806, 2.788) (1.806, 2.789) (1.806, 2.790) (1.807, 2.781) (1.807, 2.782) (1.807, 2.784) (1.807, 2.785) (1.807, 2.788) (1.807, 2.790) (1.807, 2.793) (1.808, 2.786) (1.809, 2.787) (1.809, 2.789) (1.810, 2.782) (1.810, 2.783) (1.811, 2.787) (1.811, 2.790) (1.812, 2.788) (1.812, 2.789) (1.813, 2.788) (1.813, 2.790) (1.814, 2.791) (1.815, 2.787) (1.815, 2.788) (1.815, 2.791) (1.816, 2.787) (1.816, 2.789) (1.816, 2.790) (1.816, 2.792) (1.816, 2.801) (1.817, 2.786) (1.817, 2.787) (1.817, 2.788) (1.817, 2.790) (1.817, 2.792) (1.817, 2.799) (1.817, 2.800) (1.818, 2.786) (1.818, 2.787) (1.818, 2.788) (1.818, 2.790) (1.818, 2.802) (1.818, 2.803) (1.819, 2.789) (1.819, 2.791) (1.819, 2.802) (1.819, 2.804) (1.819, 2.805) (1.820, 2.803) (1.820, 2.805) (1.820, 2.806) (1.820, 2.807) (1.821, 2.803) (1.821, 2.808) (1.822, 2.807) (1.823, 2.808) (1.823, 2.812) (1.823, 2.814) (1.824, 2.813) (1.824, 2.814) (1.825, 2.812) (1.825, 2.813) (1.826, 2.814) (1.826, 2.815) (1.827, 2.816) (1.827, 2.817) (1.827, 2.818) (1.827, 2.821) (1.827, 2.822) (1.828, 2.816) (1.828, 2.817) (1.828, 2.819) (1.828, 2.820) (1.828, 2.821) (1.828, 2.822) (1.828, 2.824) (1.829, 2.816) (1.829, 2.819) (1.829, 2.820) (1.829, 2.821) (1.829, 2.822) (1.829, 2.823) (1.829, 2.824) (1.829, 2.825) (1.829, 2.826) (1.830, 2.820) (1.830, 2.821) (1.830, 2.823) (1.830, 2.824) (1.830, 2.825) (1.830, 2.827) (1.830, 2.828) (1.831, 2.823) (1.831, 2.824) (1.831, 2.826) (1.831, 2.827) (1.831, 2.828) (1.831, 2.830) (1.832, 2.824) (1.832, 2.825) (1.832, 2.826) (1.832, 2.828) (1.833, 2.827) (1.833, 2.829) (1.834, 2.825) (1.834, 2.826) (1.835, 2.827) (1.835, 2.828) (1.836, 2.828) (1.836, 2.830) (1.836, 2.831) (1.836, 2.832) (1.837, 2.828) (1.837, 2.831) (1.837, 2.833) (1.838, 2.831) (1.838, 2.832) (1.839, 2.833) (1.840, 2.836) (1.840, 2.837) (1.840, 2.838) (1.840, 2.840) (1.841, 2.834) (1.841, 2.836) (1.841, 2.837) (1.841, 2.838) (1.841, 2.839) (1.841, 2.840) (1.841, 2.841) (1.841, 2.842) (1.842, 2.835) (1.842, 2.836) (1.842, 2.837) (1.842, 2.838) (1.842, 2.839) (1.842, 2.840) (1.842, 2.842) (1.843, 2.835) (1.843, 2.836) (1.843, 2.837) (1.843, 2.839) (1.843, 2.840) (1.844, 2.838) (1.844, 2.839) (1.845, 2.837) (1.845, 2.840) (1.845, 2.841) (1.847, 2.842) (1.848, 2.840) (1.848, 2.841) (1.848, 2.842) (1.849, 2.842) (1.849, 2.846) (1.850, 2.842) (1.850, 2.843) (1.850, 2.845) (1.850, 2.846) (1.850, 2.847) (1.852, 2.846) (1.852, 2.847) (1.852, 2.848) (1.852, 2.849) (1.852, 2.850) (1.852, 2.851) (1.853, 2.848) (1.853, 2.849) (1.853, 2.850) (1.853, 2.851) (1.854, 2.848) (1.854, 2.849) (1.854, 2.850) (1.854, 2.851) (1.854, 2.852) (1.854, 2.853) (1.855, 2.849) (1.855, 2.850) (1.855, 2.851) (1.855, 2.852) (1.855, 2.853) (1.855, 2.854) (1.855, 2.855) (1.856, 2.850) (1.856, 2.851) (1.856, 2.853) (1.856, 2.854) (1.856, 2.855) (1.857, 2.848) (1.857, 2.849) (1.857, 2.850) (1.857, 2.854) (1.857, 2.855) (1.857, 2.856) (1.858, 2.850) (1.858, 2.851) (1.858, 2.852) (1.858, 2.855) (1.859, 2.850) (1.859, 2.853) (1.859, 2.854) (1.859, 2.855) (1.860, 2.855) (1.861, 2.854) (1.862, 2.854) (1.863, 2.855) (1.865, 2.856) (1.865, 2.857) (1.865, 2.859) (1.865, 2.860) (1.866, 2.857) (1.866, 2.858) (1.866, 2.860) (1.867, 2.857) (1.867, 2.858) (1.868, 2.858) (1.868, 2.859) (1.868, 2.860) (1.868, 2.861) (1.868, 2.862) (1.869, 2.858) (1.869, 2.859) (1.869, 2.860) (1.869, 2.861) (1.869, 2.862) (1.869, 2.863) (1.869, 2.864) (1.869, 2.865) (1.869, 2.866) (1.870, 2.858) (1.870, 2.859) (1.870, 2.860) (1.870, 2.861) (1.870, 2.862) (1.870, 2.864) (1.870, 2.865) (1.870, 2.866) (1.870, 2.867) (1.870, 2.868) (1.871, 2.860) (1.871, 2.861) (1.871, 2.862) (1.871, 2.863) (1.871, 2.865) (1.871, 2.867) (1.871, 2.868) (1.871, 2.869) (1.872, 2.862) (1.872, 2.863) (1.872, 2.868) (1.872, 2.869) (1.872, 2.870) (1.872, 2.871) (1.873, 2.871) (1.874, 2.869) (1.874, 2.870) (1.874, 2.871) (1.874, 2.872) (1.875, 2.869) (1.875, 2.870) (1.875, 2.871) (1.875, 2.872) (1.875, 2.873) (1.875, 2.874) (1.876, 2.870) (1.876, 2.871) (1.876, 2.872) (1.876, 2.873) (1.876, 2.874) (1.876, 2.875) (1.876, 2.876) (1.877, 2.874) (1.877, 2.876) (1.878, 2.874) (1.879, 2.874) (1.880, 2.875) (1.880, 2.878) (1.881, 2.878) (1.882, 2.876) (1.882, 2.877) (1.882, 2.878) (1.883, 2.877) (1.883, 2.878) (1.883, 2.879) (1.883, 2.880) (1.884, 2.877) (1.884, 2.878) (1.884, 2.879) (1.884, 2.880) (1.884, 2.881) (1.884, 2.882) (1.885, 2.880) (1.885, 2.881) (1.885, 2.882) (1.885, 2.883) (1.886, 2.881) (1.886, 2.882) (1.886, 2.883) (1.887, 2.882) (1.887, 2.883) (1.887, 2.884) (1.887, 2.885) (1.888, 2.884) (1.888, 2.885) (1.888, 2.886) (1.888, 2.887) (1.889, 2.884) (1.889, 2.885) (1.889, 2.886) (1.889, 2.887) (1.890, 2.885) (1.890, 2.886) (1.890, 2.887) (1.890, 2.889) (1.890, 2.890) (1.891, 2.886) (1.891, 2.887) (1.891, 2.888) (1.891, 2.889) (1.891, 2.890) (1.892, 2.888) (1.892, 2.889) (1.892, 2.890) (1.892, 2.891) (1.893, 2.889) (1.893, 2.891) (1.893, 2.892) (1.893, 2.893) (1.894, 2.889) (1.894, 2.890) (1.895, 2.892) (1.895, 2.893) (1.896, 2.893) (1.897, 2.893) (1.898, 2.894) (1.899, 2.894) (1.899, 2.895) (1.899, 2.896) (1.900, 2.894) (1.900, 2.895) (1.900, 2.896) (1.900, 2.897) (1.901, 2.895) (1.901, 2.896) (1.901, 2.897) (1.902, 2.897) (1.903, 2.897) (1.904, 2.897) (1.904, 2.898) (1.904, 2.899) (1.904, 2.900) (1.904, 2.901) (1.904, 2.902) (1.905, 2.898) (1.905, 2.899) (1.905, 2.900) (1.905, 2.901) (1.905, 2.902) (1.905, 2.903) (1.905, 2.904) (1.906, 2.899) (1.906, 2.900) (1.906, 2.901) (1.906, 2.902) (1.906, 2.903) (1.906, 2.905) (1.907, 2.896) (1.907, 2.897) (1.907, 2.900) (1.907, 2.901) (1.907, 2.902) (1.907, 2.903) (1.907, 2.904) (1.907, 2.905) (1.908, 2.896) (1.908, 2.897) (1.908, 2.898) (1.908, 2.899) (1.908, 2.900) (1.908, 2.901) (1.908, 2.902) (1.908, 2.903) (1.908, 2.904) (1.909, 2.899) (1.909, 2.900) (1.909, 2.901) (1.909, 2.902) (1.909, 2.903) (1.909, 2.904) (1.910, 2.900) (1.910, 2.901) (1.910, 2.902) (1.911, 2.901) (1.911, 2.902) (1.912, 2.901) (1.912, 2.902) (1.912, 2.903) (1.912, 2.904) (1.913, 2.904) (1.913, 2.905) (1.913, 2.906) (1.914, 2.905) (1.914, 2.906) (1.914, 2.907) (1.914, 2.908) (1.915, 2.907) (1.915, 2.908) (1.915, 2.909) (1.916, 2.908) (1.916, 2.909) (1.916, 2.910) (1.916, 2.911) (1.917, 2.909) (1.917, 2.910) (1.917, 2.911) (1.917, 2.912) (1.918, 2.911) (1.918, 2.912) (1.918, 2.913) (1.918, 2.914) (1.919, 2.913) (1.919, 2.914) (1.920, 2.913) (1.920, 2.915) (1.920, 2.916) (1.921, 2.914) (1.921, 2.915) (1.921, 2.916) (1.921, 2.917) (1.922, 2.916) (1.922, 2.917) (1.922, 2.918) (1.923, 2.918) (1.923, 2.919) (1.924, 2.918) (1.924, 2.919) (1.924, 2.920) (1.924, 2.921) (1.925, 2.921) (1.926, 2.920) (1.926, 2.921) (1.927, 2.921) (1.927, 2.922) (1.928, 2.922) (1.929, 2.920) (1.929, 2.921) (1.929, 2.922) (1.929, 2.923) (1.929, 2.924) (1.929, 2.927) (1.929, 2.928) (1.929, 2.929) (1.930, 2.923) (1.930, 2.924) (1.930, 2.925) (1.930, 2.928) (1.930, 2.929) (1.930, 2.930) (1.930, 2.931) (1.931, 2.924) (1.931, 2.925) (1.931, 2.928) (1.931, 2.929) (1.931, 2.930) (1.931, 2.931) (1.932, 2.930) (1.932, 2.931) (1.932, 2.932) (1.933, 2.931) (1.933, 2.933) (1.933, 2.935) (1.933, 2.936) (1.933, 2.937) (1.934, 2.932) (1.934, 2.933) (1.934, 2.935) (1.934, 2.936) (1.934, 2.937) (1.934, 2.938) (1.935, 2.935) (1.935, 2.936) (1.935, 2.937) (1.936, 2.933) (1.936, 2.934) (1.936, 2.935) (1.936, 2.937) (1.937, 2.934) (1.937, 2.935) (1.937, 2.936) (1.937, 2.937) (1.938, 2.933) (1.938, 2.934) (1.938, 2.935) (1.938, 2.936) (1.939, 2.935) (1.939, 2.936) (1.939, 2.937) (1.939, 2.938) (1.940, 2.937) (1.940, 2.938) (1.941, 2.938) (1.942, 2.939) (1.942, 2.940) (1.943, 2.939) (1.943, 2.940) (1.943, 2.941) (1.944, 2.941) (1.944, 2.942) (1.944, 2.943) (1.945, 2.942) (1.945, 2.943) (1.946, 2.943) (1.946, 2.944) (1.947, 2.942) (1.947, 2.943) (1.947, 2.944) (1.947, 2.945) (1.947, 2.946) (1.948, 2.942) (1.948, 2.943) (1.948, 2.944) (1.948, 2.945) (1.948, 2.946) (1.949, 2.943) (1.949, 2.944) (1.949, 2.945) (1.949, 2.946) (1.950, 2.944) (1.950, 2.945) (1.950, 2.946) (1.950, 2.947) (1.950, 2.948) (1.951, 2.947) (1.951, 2.948) (1.952, 2.948) (1.952, 2.949) (1.953, 2.949) (1.953, 2.950) (1.954, 2.950) (1.954, 2.951) (1.955, 2.950) (1.955, 2.951) (1.955, 2.952) (1.956, 2.952) (1.956, 2.953) (1.956, 2.954) (1.957, 2.952) (1.957, 2.953) (1.957, 2.954) (1.957, 2.955) (1.957, 2.956) (1.958, 2.956) (1.958, 2.957) (1.959, 2.956) (1.959, 2.957) (1.959, 2.958) (1.960, 2.956) (1.960, 2.957) (1.960, 2.958) (1.961, 2.957) (1.961, 2.958) (1.961, 2.959) (1.962, 2.958) (1.962, 2.959) (1.963, 2.959) (1.963, 2.960) (1.964, 2.959) (1.964, 2.960) (1.964, 2.961) (1.965, 2.960) (1.965, 2.961) (1.965, 2.962) (1.966, 2.961) (1.966, 2.962) (1.967, 2.961) (1.967, 2.962) (1.967, 2.963) (1.967, 2.964) (1.968, 2.961) (1.968, 2.962) (1.968, 2.963) (1.968, 2.964) (1.968, 2.965) (1.969, 2.963) (1.969, 2.964) (1.969, 2.965) (1.969, 2.966) (1.970, 2.965) (1.970, 2.966) (1.970, 2.967) (1.971, 2.966) (1.971, 2.967) (1.971, 2.968) (1.972, 2.967) (1.972, 2.968) (1.972, 2.969) (1.973, 2.968) (1.973, 2.969) (1.974, 2.969) (1.974, 2.970) (1.974, 2.971) (1.975, 2.969) (1.975, 2.970) (1.975, 2.971) (1.975, 2.972) (1.975, 2.973) (1.976, 2.970) (1.976, 2.971) (1.976, 2.972) (1.976, 2.973) (1.976, 2.974) (1.977, 2.973) (1.977, 2.974) (1.978, 2.974) (1.978, 2.975) (1.978, 2.976) (1.978, 2.977) (1.979, 2.975) (1.979, 2.976) (1.979, 2.977) (1.979, 2.978) (1.980, 2.978) (1.980, 2.979) (1.980, 2.980) (1.981, 2.979) (1.981, 2.980) (1.981, 2.981) (1.981, 2.982) (1.982, 2.981) (1.982, 2.982) (1.982, 2.983) (1.983, 2.982) (1.983, 2.983) (1.984, 2.983) (1.984, 2.984) (1.985, 2.983) (1.985, 2.984) (1.985, 2.985) (1.986, 2.985) (1.986, 2.986) (1.987, 2.985) (1.987, 2.986) (1.987, 2.987) (1.988, 2.987) (1.988, 2.988) (1.989, 2.987) (1.989, 2.988) (1.989, 2.989) (1.990, 2.988) (1.990, 2.989) (1.990, 2.990) (1.991, 2.990) (1.991, 2.991) (1.992, 2.991) (1.992, 2.992) (1.993, 2.992) (1.993, 2.993) (1.994, 2.992) (1.994, 2.993) (1.995, 2.992) (1.995, 2.993) (1.995, 2.994) (1.996, 2.993) (1.996, 2.994) (1.996, 2.995) (1.997, 2.994) (1.997, 2.995) (1.998, 2.994) (1.998, 2.995) (1.998, 2.996) (1.999, 2.995) (1.999, 2.996) (1.999, 2.997) (2.000, 2.996) (2.000, 2.997) (2.000, 2.998) (2.000, 2.999) (2., 2.996) (2.001, 2.997) (2.001, 2.998) (2.001, 2.999) (2.001, 3.000) (2.001, 3.) (2.002, 2.998) (2.002, 2.999) (2.002, 3.000) (2.002, 3.) (2.002, 3.001) (2.003, 3.000) (2.003, 3.) (2.003, 3.001) (2.003, 3.002) (2.004, 3.002) (2.004, 3.003) (2.004, 3.004) (2.005, 3.003) (2.005, 3.004) (2.005, 3.005) (2.006, 3.005) (2.006, 3.006) (2.006, 3.007) (2.007, 3.006) (2.007, 3.007) (2.008, 3.007) (2.008, 3.008) (2.009, 3.008) (2.009, 3.009) (2.010, 3.009) (2.010, 3.010) (2.010, 3.011) (2.011, 3.011) (2.012, 3.012) (2.012, 3.013) (2.013, 3.012) (2.013, 3.013) (2.013, 3.014) (2.014, 3.013) (2.014, 3.014) (2.014, 3.015) (2.015, 3.014) (2.015, 3.015) (2.015, 3.016) (2.016, 3.015) (2.016, 3.016) (2.016, 3.017) (2.017, 3.016) (2.017, 3.017) (2.017, 3.018) (2.017, 3.019) (2.018, 3.019) (2.018, 3.020) (2.019, 3.020) (2.019, 3.021) (2.020, 3.020) (2.020, 3.021) (2.020, 3.022) (2.021, 3.022) (2.021, 3.023) (2.022, 3.022) (2.022, 3.023) (2.022, 3.024) (2.023, 3.023) (2.023, 3.024) (2.023, 3.025) (2.024, 3.024) (2.024, 3.025) (2.024, 3.026) (2.025, 3.025) (2.025, 3.026) (2.025, 3.027) (2.026, 3.026) (2.026, 3.027) (2.026, 3.028) (2.027, 3.027) (2.027, 3.028) (2.027, 3.029) (2.028, 3.028) (2.028, 3.029) (2.029, 3.029) (2.029, 3.030) (2.030, 3.030) (2.030, 3.031) (2.031, 3.031) (2.031, 3.032) (2.032, 3.032) (2.032, 3.033) (2.032, 3.034) (2.033, 3.034) (2.033, 3.035) (2.034, 3.034) (2.034, 3.035) (2.035, 3.035) (2.035, 3.036) (2.036, 3.036) (2.036, 3.037) (2.036, 3.038) (2.037, 3.037) (2.037, 3.038) (2.038, 3.038) (2.038, 3.039) (2.039, 3.038) (2.039, 3.039) (2.039, 3.040) (2.040, 3.040) (2.040, 3.041) (2.041, 3.041) (2.041, 3.042) (2.042, 3.042) (2.042, 3.043) (2.043, 3.043) (2.043, 3.044) (2.044, 3.044) (2.044, 3.045) (2.045, 3.045) (2.045, 3.046) (2.045, 3.047) (2.046, 3.046) (2.046, 3.047) (2.046, 3.048) (2.047, 3.047) (2.047, 3.048) (2.047, 3.049) (2.048, 3.049) (2.048, 3.050) (2.049, 3.050) (2.049, 3.051) (2.049, 3.052) (2.050, 3.051) (2.050, 3.052) (2.050, 3.053) (2.051, 3.053) (2.051, 3.054) (2.052, 3.054) (2.052, 3.055) (2.053, 3.055) (2.053, 3.056) (2.054, 3.056) (2.054, 3.057) (2.055, 3.057) (2.055, 3.058) (2.056, 3.058) (2.056, 3.059) (2.056, 3.060) (2.057, 3.060) (2.057, 3.061) (2.058, 3.061) (2.058, 3.062) (2.059, 3.062) (2.059, 3.063) (2.059, 3.064) (2.060, 3.063) (2.060, 3.064) (2.060, 3.065) (2.061, 3.065) (2.061, 3.066) (2.062, 3.066) (2.062, 3.067) (2.063, 3.067) (2.063, 3.068) (2.063, 3.069) (2.064, 3.068) (2.064, 3.069) (2.064, 3.070) (2.065, 3.070) (2.065, 3.071) (2.066, 3.071) (2.066, 3.072) (2.067, 3.072) (2.067, 3.073) (2.068, 3.073) (2.068, 3.074) (2.068, 3.075) (2.069, 3.075) (2.069, 3.076) (2.070, 3.076) (2.070, 3.077) (2.071, 3.077) (2.071, 3.078) (2.072, 3.078) (2.072, 3.079) (2.073, 3.079) (2.073, 3.080) (2.074, 3.079) (2.074, 3.080) (2.074, 3.081) (2.075, 3.080) (2.075, 3.081) (2.075, 3.082) (2.076, 3.082) (2.076, 3.083) (2.077, 3.083) (2.077, 3.084) (2.077, 3.085) (2.078, 3.084) (2.078, 3.085) (2.079, 3.085) (2.079, 3.086) (2.080, 3.086) (2.080, 3.087) (2.081, 3.087) (2.081, 3.088) (2.082, 3.088) (2.082, 3.089) (2.082, 3.090) (2.083, 3.090) (2.083, 3.091) (2.084, 3.091) (2.084, 3.092) (2.085, 3.092) (2.085, 3.093) (2.085, 3.094) (2.086, 3.093) (2.086, 3.094) (2.087, 3.094) (2.087, 3.095) (2.087, 3.096) (2.088, 3.095) (2.088, 3.096) (2.088, 3.097) (2.089, 3.097) (2.089, 3.098) (2.090, 3.098) (2.090, 3.099) (2.091, 3.099) (2.091, 3.100) (2.092, 3.099) (2.092, 3.100) (2.093, 3.100) (2.093, 3.101) (2.093, 3.102) (2.094, 3.101) (2.094, 3.102) (2.094, 3.103) (2.095, 3.103) (2.095, 3.104) (2.096, 3.104) (2.096, 3.105) (2.096, 3.106) (2.097, 3.105) (2.097, 3.106) (2.098, 3.106) (2.098, 3.107) (2.098, 3.108) (2.099, 3.107) (2.099, 3.108) (2.099, 3.109) (2.100, 3.109) (2.100, 3.110) (2.101, 3.109) (2.101, 3.110) (2.101, 3.111) (2.102, 3.111) (2.102, 3.112) (2.103, 3.112) (2.103, 3.113) (2.104, 3.113) (2.104, 3.114) (2.105, 3.114) (2.105, 3.115) (2.106, 3.115) (2.106, 3.116) (2.107, 3.116) (2.107, 3.117) (2.107, 3.118) (2.108, 3.117) (2.108, 3.118) (2.108, 3.119) (2.109, 3.119) (2.109, 3.120) (2.110, 3.120) (2.110, 3.121) (2.111, 3.121) (2.111, 3.122) (2.112, 3.122) (2.112, 3.123) (2.113, 3.123) (2.113, 3.124) (2.114, 3.124) (2.114, 3.125) (2.115, 3.125) (2.115, 3.126) (2.116, 3.126) (2.116, 3.127) (2.117, 3.127) (2.117, 3.128) (2.118, 3.128) (2.118, 3.129) (2.119, 3.129) (2.119, 3.130) (2.120, 3.130) (2.120, 3.131) (2.121, 3.131) (2.121, 3.132) (2.122, 3.132) (2.122, 3.133)};
\end{axis}
\end{tikzpicture}

%% file: fig-melaluiprimepower.tex
\begin{tabular}{ccc}
\begin{tikzpicture}
\begin{axis}[
	xmin=8.138739285,
	xmax=14.19395072,
	ymin=-.3355556818,
	ymax=-.2909203926,
    xlabel near ticks,
    ylabel near ticks,
	samples=100,
	xlabel=$\ln x$,
	ylabel={$\ln\left[\bd_x\rA^{-1}\left(\mathbb{N}_{\geqslant 2}\right)\right]$},
	width=8cm,
    height=6cm,
    clip=false,grid=major,clip marker paths=true
]

\addplot[color=newblue,mark=x] plot coordinates {(8.51719, -.294640) (8.69951, -.296834) (8.85367, -.296291) (8.98719, -.297396) (9.10498, -.298406) (9.21036, -.298676) (9.30565, -.300369) (9.39266, -.301217) (9.47270, -.301001) (9.54681, -.302168) (9.61581, -.302908) (9.68034, -.302457) (9.74097, -.303333) (9.79813, -.303811) (9.85219, -.304668) (9.90349, -.304828) (9.95228, -.305425) (9.99880, -.305907) (10.0432, -.306171) (10.0858, -.306355) (10.1266, -.306308) (10.1659, -.306681) (10.2036, -.306575) (10.2400, -.307594) (10.2751, -.307745) (10.3090, -.308021) (10.3417, -.308104) (10.3735, -.308650) (10.4043, -.308957) (10.4341, -.309086) (10.4631, -.309324) (10.4913, -.309436) (10.5187, -.309209) (10.5453, -.309462) (10.5713, -.309735) (10.5966, -.309689) (10.6213, -.309877) (10.6454, -.310024) (10.6690, -.309849) (10.6919, -.310269) (10.7144, -.310458) (10.7364, -.310580) (10.7579, -.310870) (10.7790, -.310893) (10.7996, -.311056) (10.8198, -.311319) (10.8396, -.311546) (10.8590, -.311633) (10.8780, -.311717) (10.8967, -.311949) (10.9151, -.312124) (10.9331, -.312292) (10.9508, -.312286) (10.9682, -.312352) (10.9853, -.312531) (11.0021, -.312635) (11.0186, -.312714) (11.0349, -.312900) (11.0509, -.312973) (11.0666, -.312957) (11.0821, -.313321) (11.0974, -.313403) (11.1124, -.313607) (11.1273, -.313483) (11.1419, -.313758) (11.1563, -.314144) (11.1704, -.314210) (11.1844, -.314178) (11.1982, -.314316) (11.2118, -.314378) (11.2252, -.314455) (11.2385, -.314477) (11.2516, -.314604) (11.2645, -.314659) (11.2772, -.314711) (11.2898, -.314796) (11.3022, -.314948) (11.3145, -.314994) (11.3266, -.315123) (11.3386, -.315200) (11.3504, -.315323) (11.3621, -.315492) (11.3737, -.315482) (11.3851, -.315629) (11.3964, -.315928) (11.4076, -.315852) (11.4186, -.315961) (11.4295, -.316186) (11.4404, -.316112) (11.4511, -.316301) (11.4616, -.316515) (11.4721, -.316367) (11.4825, -.316378) (11.4927, -.316585) (11.5029, -.316691) (11.5130, -.316740) (11.5229, -.316869) (11.5327, -.316807) (11.5425, -.316974) (11.5521, -.317151) (11.5617, -.317114) (11.5712, -.317156) (11.5806, -.317262) (11.5899, -.317442) (11.5991, -.317530) (11.6082, -.317454) (11.6173, -.317541) (11.6263, -.317651) (11.6351, -.317745) (11.6440, -.317888) (11.6527, -.317968) (11.6613, -.318011) (11.6699, -.318030) (11.6784, -.318048) (11.6869, -.318228) (11.6952, -.318268) (11.7035, -.318238) (11.7118, -.318254) (11.7199, -.318360) (11.7280, -.318330) (11.7361, -.318389) (11.7440, -.318382) (11.7519, -.318396) (11.7598, -.318400) (11.7676, -.318456) (11.7753, -.318596) (11.7830, -.318618) (11.7906, -.318715) (11.7981, -.318767) (11.8056, -.318880) (11.8130, -.318993) (11.8204, -.319010) (11.8277, -.319030) (11.8350, -.319048) (11.8422, -.319136) (11.8494, -.319144) (11.8565, -.319160) (11.8636, -.319167) (11.8706, -.319262) (11.8776, -.319325) (11.8845, -.319379) (11.8914, -.319431) (11.8982, -.319419) (11.9050, -.319415) (11.9117, -.319521) (11.9184, -.319499) (11.9250, -.319540) (11.9316, -.319616) (11.9382, -.319665) (11.9447, -.319714) (11.9512, -.319796) (11.9576, -.319809) (11.9640, -.319898) (11.9704, -.319935) (11.9767, -.319971) (11.9829, -.320025) (11.9892, -.320059) (11.9954, -.320146) (12.0015, -.320171) (12.0076, -.320222) (12.0137, -.320247) (12.0197, -.320321) (12.0257, -.320395) (12.0317, -.320460) (12.0377, -.320515) (12.0436, -.320554) (12.0494, -.320519) (12.0552, -.320566) (12.0610, -.320540) (12.0668, -.320601) (12.0725, -.320646) (12.0782, -.320707) (12.0839, -.320696) (12.0895, -.320824) (12.0951, -.320890) (12.1007, -.320864) (12.1063, -.320929) (12.1118, -.321016) (12.1172, -.321011) (12.1227, -.321060) (12.1281, -.321100) (12.1335, -.321079) (12.1389, -.321126) (12.1442, -.321137) (12.1495, -.321124) (12.1548, -.321097) (12.1600, -.321158) (12.1653, -.321232) (12.1704, -.321240) (12.1756, -.321257) (12.1808, -.321328) (12.1859, -.321366) (12.1910, -.321436) (12.1960, -.321444) (12.2011, -.321453) (12.2061, -.321549) (12.2111, -.321563) (12.2160, -.321584) (12.2210, -.321611) (12.2259, -.321672) (12.2308, -.321745) (12.2356, -.321771) (12.2405, -.321797) (12.2453, -.321808) (12.2501, -.321861) (12.2549, -.321833) (12.2596, -.321877) (12.2643, -.321909) (12.2690, -.321966) (12.2737, -.321926) (12.2784, -.321930) (12.2830, -.321973) (12.2877, -.322035) (12.2923, -.322078) (12.2968, -.322126) (12.3014, -.322091) (12.3059, -.322133) (12.3104, -.322167) (12.3149, -.322196) (12.3194, -.322286) (12.3239, -.322246) (12.3283, -.322304) (12.3327, -.322331) (12.3371, -.322377) (12.3415, -.322373) (12.3458, -.322363) (12.3502, -.322384) (12.3545, -.322434) (12.3588, -.322436) (12.3631, -.322557) (12.3673, -.322541) (12.3716, -.322624) (12.3758, -.322672) (12.3800, -.322714) (12.3842, -.322699) (12.3884, -.322740) (12.3926, -.322758) (12.3967, -.322747) (12.4008, -.322776) (12.4049, -.322788) (12.4090, -.322777) (12.4131, -.322852) (12.4171, -.322869) (12.4212, -.322925) (12.4252, -.322853) (12.4292, -.322892) (12.4332, -.322936) (12.4372, -.322931) (12.4411, -.322914) (12.4451, -.322942) (12.4490, -.322986) (12.4529, -.322992) (12.4568, -.323061) (12.4607, -.323124) (12.4646, -.323139) (12.4684, -.323155) (12.4723, -.323207) (12.4761, -.323233) (12.4799, -.323264) (12.4837, -.323262) (12.4875, -.323293) (12.4913, -.323297) (12.4950, -.323331) (12.4987, -.323300) (12.5025, -.323365) (12.5062, -.323342) (12.5099, -.323347) (12.5136, -.323370) (12.5172, -.323409) (12.5209, -.323458) (12.5245, -.323467) (12.5282, -.323460) (12.5318, -.323503) (12.5354, -.323570) (12.5390, -.323598) (12.5425, -.323595) (12.5461, -.323594) (12.5497, -.323595) (12.5532, -.323633) (12.5567, -.323699) (12.5602, -.323667) (12.5637, -.323707) (12.5672, -.323763) (12.5707, -.323804) (12.5742, -.323844) (12.5776, -.323850) (12.5811, -.323843) (12.5845, -.323887) (12.5879, -.323879) (12.5913, -.323917) (12.5947, -.323905) (12.5981, -.323922) (12.6015, -.323941) (12.6048, -.323966) (12.6082, -.324027) (12.6115, -.324042) (12.6149, -.324089) (12.6182, -.324140) (12.6215, -.324154) (12.6248, -.324169) (12.6281, -.324141) (12.6313, -.324139) (12.6346, -.324134) (12.6379, -.324194) (12.6411, -.324216) (12.6443, -.324252) (12.6475, -.324276) (12.6508, -.324320) (12.6540, -.324350) (12.6571, -.324346) (12.6603, -.324376) (12.6635, -.324399) (12.6667, -.324394) (12.6698, -.324433) (12.6729, -.324450) (12.6761, -.324450) (12.6792, -.324489) (12.6823, -.324497) (12.6854, -.324518) (12.6885, -.324573) (12.6916, -.324563) (12.6947, -.324584) (12.6977, -.324663) (12.7008, -.324663) (12.7038, -.324729) (12.7068, -.324715) (12.7099, -.324756) (12.7129, -.324762) (12.7159, -.324803) (12.7189, -.324818) (12.7219, -.324841) (12.7249, -.324857) (12.7278, -.324863) (12.7308, -.324891) (12.7338, -.324909) (12.7367, -.324944) (12.7396, -.324967) (12.7426, -.324981) (12.7455, -.325016) (12.7484, -.325030) (12.7513, -.325043) (12.7542, -.325074) (12.7571, -.325096) (12.7600, -.325125) (12.7628, -.325135) (12.7657, -.325136) (12.7685, -.325162) (12.7714, -.325164) (12.7742, -.325185) (12.7771, -.325190) (12.7799, -.325187) (12.7827, -.325205) (12.7855, -.325265) (12.7883, -.325297) (12.7911, -.325309) (12.7939, -.325341) (12.7966, -.325388) (12.7994, -.325413) (12.8022, -.325413) (12.8049, -.325421) (12.8077, -.325431) (12.8104, -.325450) (12.8131, -.325424) (12.8158, -.325437) (12.8186, -.325475) (12.8213, -.325517) (12.8240, -.325521) (12.8266, -.325540) (12.8293, -.325533) (12.8320, -.325571) (12.8347, -.325568) (12.8373, -.325612) (12.8400, -.325623) (12.8426, -.325628) (12.8453, -.325628) (12.8479, -.325628) (12.8506, -.325672) (12.8532, -.325687) (12.8558, -.325726) (12.8584, -.325726) (12.8610, -.325734) (12.8636, -.325734) (12.8662, -.325784) (12.8688, -.325812) (12.8713, -.325848) (12.8739, -.325837) (12.8765, -.325851) (12.8790, -.325885) (12.8816, -.325881) (12.8841, -.325878) (12.8866, -.325899) (12.8892, -.325932) (12.8917, -.325953) (12.8942, -.325997) (12.8967, -.325984) (12.8992, -.325996) (12.9017, -.326027) (12.9042, -.326068) (12.9067, -.326099) (12.9092, -.326125) (12.9116, -.326158) (12.9141, -.326185) (12.9166, -.326223) (12.9190, -.326212) (12.9215, -.326204) (12.9239, -.326209) (12.9263, -.326250) (12.9288, -.326251) (12.9312, -.326277) (12.9336, -.326316) (12.9360, -.326304) (12.9384, -.326313) (12.9408, -.326351) (12.9432, -.326366) (12.9456, -.326345) (12.9480, -.326370) (12.9504, -.326402) (12.9528, -.326420) (12.9551, -.326444) (12.9575, -.326426) (12.9598, -.326434) (12.9622, -.326455) (12.9645, -.326502) (12.9669, -.326549) (12.9692, -.326570) (12.9715, -.326571) (12.9739, -.326604) (12.9762, -.326623) (12.9785, -.326623) (12.9808, -.326627) (12.9831, -.326606) (12.9854, -.326617) (12.9877, -.326650) (12.9900, -.326647) (12.9923, -.326652) (12.9945, -.326646) (12.9968, -.326677) (12.9991, -.326696) (13.0013, -.326697) (13.0036, -.326692) (13.0058, -.326717) (13.0081, -.326753) (13.0103, -.326775) (13.0125, -.326804) (13.0148, -.326832) (13.0170, -.326814) (13.0192, -.326778) (13.0214, -.326785) (13.0236, -.326828) (13.0259, -.326860) (13.0281, -.326910) (13.0302, -.326903) (13.0324, -.326922) (13.0346, -.326947) (13.0368, -.326944) (13.0390, -.326983) (13.0412, -.326977) (13.0433, -.327008) (13.0455, -.327011) (13.0476, -.327015) (13.0498, -.327020) (13.0519, -.327036) (13.0541, -.327048) (13.0562, -.327048) (13.0584, -.327030) (13.0605, -.327040) (13.0626, -.327045) (13.0647, -.327063) (13.0669, -.327090) (13.0690, -.327095) (13.0711, -.327087) (13.0732, -.327098) (13.0753, -.327105) (13.0774, -.327151) (13.0795, -.327142) (13.0815, -.327145) (13.0836, -.327137) (13.0857, -.327154) (13.0878, -.327148) (13.0898, -.327148) (13.0919, -.327173) (13.0940, -.327184) (13.0960, -.327197) (13.0981, -.327213) (13.1001, -.327201) (13.1022, -.327222) (13.1042, -.327249) (13.1062, -.327266) (13.1083, -.327277) (13.1103, -.327278) (13.1123, -.327273) (13.1143, -.327278) (13.1163, -.327306) (13.1184, -.327295) (13.1204, -.327324) (13.1224, -.327333) (13.1244, -.327352) (13.1264, -.327364) (13.1283, -.327387) (13.1303, -.327380) (13.1323, -.327394) (13.1343, -.327402) (13.1363, -.327392) (13.1382, -.327409) (13.1402, -.327421) (13.1422, -.327434) (13.1441, -.327456) (13.1461, -.327452) (13.1480, -.327478) (13.1500, -.327487) (13.1519, -.327491) (13.1539, -.327517) (13.1558, -.327503) (13.1577, -.327496) (13.1597, -.327534) (13.1616, -.327538) (13.1635, -.327561) (13.1654, -.327571) (13.1673, -.327596) (13.1692, -.327623) (13.1712, -.327656) (13.1731, -.327664) (13.1750, -.327661) (13.1769, -.327675) (13.1787, -.327699) (13.1806, -.327718) (13.1825, -.327710) (13.1844, -.327715) (13.1863, -.327715) (13.1882, -.327735) (13.1900, -.327736) (13.1919, -.327735) (13.1938, -.327731) (13.1956, -.327738) (13.1975, -.327767) (13.1993, -.327807) (13.2012, -.327824) (13.2030, -.327833) (13.2049, -.327821) (13.2067, -.327843) (13.2085, -.327875) (13.2104, -.327881) (13.2122, -.327876) (13.2140, -.327889) (13.2159, -.327882) (13.2177, -.327900) (13.2195, -.327904) (13.2213, -.327900) (13.2231, -.327926) (13.2249, -.327913) (13.2267, -.327933) (13.2285, -.327975) (13.2303, -.327958) (13.2321, -.327979) (13.2339, -.327985) (13.2357, -.327978) (13.2375, -.327976) (13.2393, -.328000) (13.2410, -.327996) (13.2428, -.328033) (13.2446, -.328035) (13.2463, -.328053) (13.2481, -.328065) (13.2499, -.328086) (13.2516, -.328104) (13.2534, -.328117) (13.2551, -.328119) (13.2569, -.328131) (13.2586, -.328121) (13.2604, -.328139) (13.2621, -.328158) (13.2639, -.328160) (13.2656, -.328150) (13.2673, -.328162) (13.2691, -.328178) (13.2708, -.328197) (13.2725, -.328225) (13.2742, -.328239) (13.2759, -.328265) (13.2777, -.328257) (13.2794, -.328264) (13.2811, -.328272) (13.2828, -.328276) (13.2845, -.328258) (13.2862, -.328274) (13.2879, -.328311) (13.2896, -.328330) (13.2913, -.328361) (13.2929, -.328368) (13.2946, -.328378) (13.2963, -.328397) (13.2980, -.328421) (13.2997, -.328422) (13.3013, -.328439) (13.3030, -.328449) (13.3047, -.328453) (13.3064, -.328480) (13.3080, -.328493) (13.3097, -.328480) (13.3113, -.328487) (13.3130, -.328490) (13.3146, -.328494) (13.3163, -.328507) (13.3179, -.328522) (13.3196, -.328532) (13.3212, -.328526) (13.3229, -.328539) (13.3245, -.328558) (13.3261, -.328565) (13.3278, -.328569) (13.3294, -.328580) (13.3310, -.328594) (13.3326, -.328590) (13.3342, -.328585) (13.3359, -.328610) (13.3375, -.328614) (13.3391, -.328611) (13.3407, -.328621) (13.3423, -.328633) (13.3439, -.328651) (13.3455, -.328653) (13.3471, -.328664) (13.3487, -.328686) (13.3503, -.328705) (13.3519, -.328736) (13.3535, -.328751) (13.3551, -.328779) (13.3566, -.328783) (13.3582, -.328804) (13.3598, -.328807) (13.3614, -.328794) (13.3630, -.328810) (13.3645, -.328825) (13.3661, -.328842) (13.3677, -.328885) (13.3692, -.328899) (13.3708, -.328899) (13.3723, -.328908) (13.3739, -.328940) (13.3755, -.328953) (13.3770, -.328951) (13.3786, -.328960) (13.3801, -.328969) (13.3816, -.328974) (13.3832, -.328979) (13.3847, -.328989) (13.3863, -.329013) (13.3878, -.329042) (13.3893, -.329049) (13.3909, -.329074) (13.3924, -.329067) (13.3939, -.329074) (13.3954, -.329083) (13.3970, -.329097) (13.3985, -.329088) (13.4000, -.329089) (13.4015, -.329089) (13.4030, -.329107) (13.4045, -.329128) (13.4060, -.329146) (13.4075, -.329147) (13.4090, -.329153) (13.4105, -.329171) (13.4120, -.329190) (13.4135, -.329206) (13.4150, -.329207) (13.4165, -.329213) (13.4180, -.329213) (13.4195, -.329217) (13.4210, -.329224) (13.4225, -.329227) (13.4239, -.329242) (13.4254, -.329252) (13.4269, -.329286) (13.4284, -.329288) (13.4298, -.329282) (13.4313, -.329296) (13.4328, -.329313) (13.4343, -.329322) (13.4357, -.329304) (13.4372, -.329317) (13.4386, -.329317) (13.4401, -.329338) (13.4415, -.329328) (13.4430, -.329329) (13.4444, -.329324) (13.4459, -.329363) (13.4473, -.329366) (13.4488, -.329363) (13.4502, -.329377) (13.4517, -.329400) (13.4531, -.329392) (13.4545, -.329404) (13.4560, -.329414) (13.4574, -.329424) (13.4588, -.329421) (13.4603, -.329438) (13.4617, -.329460) (13.4631, -.329453) (13.4645, -.329452) (13.4660, -.329446) (13.4674, -.329449) (13.4688, -.329446) (13.4702, -.329448) (13.4716, -.329467) (13.4730, -.329477) (13.4744, -.329481) (13.4758, -.329470) (13.4772, -.329486) (13.4786, -.329495) (13.4800, -.329509) (13.4814, -.329509) (13.4828, -.329535) (13.4842, -.329559) (13.4856, -.329573) (13.4870, -.329581) (13.4884, -.329585) (13.4898, -.329575) (13.4912, -.329577) (13.4925, -.329577) (13.4939, -.329577) (13.4953, -.329599) (13.4967, -.329594) (13.4981, -.329609) (13.4994, -.329621) (13.5008, -.329631) (13.5022, -.329649) (13.5035, -.329656) (13.5049, -.329660) (13.5063, -.329667) (13.5076, -.329664) (13.5090, -.329669) (13.5103, -.329684) (13.5117, -.329692) (13.5131, -.329706) (13.5144, -.329708) (13.5158, -.329719) (13.5171, -.329727) (13.5185, -.329727) (13.5198, -.329742) (13.5211, -.329759) (13.5225, -.329778) (13.5238, -.329785) (13.5252, -.329806) (13.5265, -.329817) (13.5278, -.329823) (13.5292, -.329834) (13.5305, -.329823) (13.5318, -.329826) (13.5331, -.329842) (13.5345, -.329836) (13.5358, -.329844) (13.5371, -.329861) (13.5384, -.329866) (13.5398, -.329870) (13.5411, -.329865) (13.5424, -.329863) (13.5437, -.329863) (13.5450, -.329869) (13.5463, -.329861) (13.5476, -.329880) (13.5489, -.329872) (13.5502, -.329872) (13.5515, -.329890) (13.5528, -.329887) (13.5541, -.329899) (13.5554, -.329899) (13.5567, -.329902) (13.5580, -.329918) (13.5593, -.329933) (13.5606, -.329933) (13.5619, -.329950) (13.5632, -.329963) (13.5645, -.329965) (13.5658, -.329961) (13.5670, -.329965) (13.5683, -.329977) (13.5696, -.329983) (13.5709, -.329998) (13.5722, -.330025) (13.5734, -.330034) (13.5747, -.330039) (13.5760, -.330062) (13.5773, -.330073) (13.5785, -.330066) (13.5798, -.330073) (13.5811, -.330069) (13.5823, -.330083) (13.5836, -.330073) (13.5848, -.330087) (13.5861, -.330097) (13.5874, -.330104) (13.5886, -.330114) (13.5899, -.330115) (13.5911, -.330122) (13.5924, -.330137) (13.5936, -.330123) (13.5949, -.330144) (13.5961, -.330155) (13.5974, -.330165) (13.5986, -.330174) (13.5998, -.330172) (13.6011, -.330189) (13.6023, -.330203) (13.6036, -.330221) (13.6048, -.330221) (13.6060, -.330208) (13.6073, -.330204) (13.6085, -.330219) (13.6097, -.330215) (13.6109, -.330218) (13.6122, -.330211) (13.6134, -.330218) (13.6146, -.330222) (13.6158, -.330233) (13.6171, -.330225) (13.6183, -.330229) (13.6195, -.330235) (13.6207, -.330239) (13.6219, -.330242) (13.6231, -.330254) (13.6244, -.330253) (13.6256, -.330249) (13.6268, -.330261) (13.6280, -.330267) (13.6292, -.330276) (13.6304, -.330274) (13.6316, -.330285) (13.6328, -.330275) (13.6340, -.330288) (13.6352, -.330306) (13.6364, -.330311) (13.6376, -.330329) (13.6388, -.330349) (13.6400, -.330354) (13.6412, -.330384) (13.6423, -.330386) (13.6435, -.330389) (13.6447, -.330382) (13.6459, -.330395) (13.6471, -.330391) (13.6483, -.330402) (13.6495, -.330418) (13.6506, -.330409) (13.6518, -.330407) (13.6530, -.330421) (13.6542, -.330428) (13.6553, -.330420) (13.6565, -.330423) (13.6577, -.330441) (13.6589, -.330453) (13.6600, -.330467) (13.6612, -.330464) (13.6624, -.330474) (13.6635, -.330477) (13.6647, -.330484) (13.6658, -.330475) (13.6670, -.330485) (13.6682, -.330477) (13.6693, -.330498) (13.6705, -.330498) (13.6716, -.330495) (13.6728, -.330503) (13.6739, -.330503) (13.6751, -.330517) (13.6762, -.330516) (13.6774, -.330510) (13.6785, -.330533) (13.6797, -.330539) (13.6808, -.330558) (13.6820, -.330555) (13.6831, -.330560) (13.6843, -.330571) (13.6854, -.330583) (13.6865, -.330576) (13.6877, -.330580) (13.6888, -.330577) (13.6899, -.330592) (13.6911, -.330598) (13.6922, -.330610) (13.6933, -.330627) (13.6945, -.330624) (13.6956, -.330636) (13.6967, -.330637) (13.6979, -.330648) (13.6990, -.330662) (13.7001, -.330670) (13.7012, -.330680) (13.7023, -.330690) (13.7035, -.330695) (13.7046, -.330693) (13.7057, -.330695) (13.7068, -.330708) (13.7079, -.330716) (13.7090, -.330722) (13.7102, -.330732) (13.7113, -.330736) (13.7124, -.330755) (13.7135, -.330761) (13.7146, -.330779) (13.7157, -.330775) (13.7168, -.330786) (13.7179, -.330779) (13.7190, -.330791) (13.7201, -.330797) (13.7212, -.330801) (13.7223, -.330828) (13.7234, -.330825) (13.7245, -.330822) (13.7256, -.330850) (13.7267, -.330846) (13.7278, -.330850) (13.7289, -.330862) (13.7300, -.330861) (13.7310, -.330862) (13.7321, -.330878) (13.7332, -.330890) (13.7343, -.330900) (13.7354, -.330907) (13.7365, -.330906) (13.7375, -.330931) (13.7386, -.330921) (13.7397, -.330921) (13.7408, -.330922) (13.7419, -.330940) (13.7429, -.330936) (13.7440, -.330947) (13.7451, -.330952) (13.7462, -.330958) (13.7472, -.330968) (13.7483, -.330958) (13.7494, -.330968) (13.7504, -.330968) (13.7515, -.330984) (13.7526, -.330984) (13.7536, -.331002) (13.7547, -.331004) (13.7558, -.331003) (13.7568, -.331014) (13.7579, -.331024) (13.7589, -.331029) (13.7600, -.331053) (13.7611, -.331066) (13.7621, -.331078) (13.7632, -.331085) (13.7642, -.331089) (13.7653, -.331088) (13.7663, -.331105) (13.7674, -.331100) (13.7684, -.331113) (13.7695, -.331120) (13.7705, -.331128) (13.7716, -.331145) (13.7726, -.331149) (13.7736, -.331155) (13.7747, -.331159) (13.7757, -.331172) (13.7768, -.331174) (13.7778, -.331195) (13.7788, -.331201) (13.7799, -.331209) (13.7809, -.331217) (13.7820, -.331217) (13.7830, -.331236) (13.7840, -.331243) (13.7851, -.331249) (13.7861, -.331245) (13.7871, -.331241) (13.7881, -.331252) (13.7892, -.331276) (13.7902, -.331273) (13.7912, -.331276) (13.7922, -.331290) (13.7933, -.331307) (13.7943, -.331311) (13.7953, -.331315) (13.7963, -.331337) (13.7973, -.331336) (13.7984, -.331348) (13.7994, -.331365) (13.8004, -.331375) (13.8014, -.331378) (13.8024, -.331376) (13.8034, -.331385) (13.8044, -.331390) (13.8055, -.331401) (13.8065, -.331404) (13.8075, -.331397) (13.8085, -.331410) (13.8095, -.331428) (13.8105, -.331433) (13.8115, -.331444) (13.8125, -.331457) (13.8135, -.331444) (13.8145, -.331447) (13.8155, -.331460)};

\addplot[domain=8.51719:13.8125,ultra thick,newpurple] {-0.638407522714256e-2*x-.243451949157420};

\end{axis}
\end{tikzpicture}&&\begin{tikzpicture}
\begin{axis}[
	xmin=-66071.42857,
	xmax=1071071.429,
	ymin=.7151825000,
	ymax=.7474925000,
    xlabel near ticks,
    ylabel near ticks,
	samples=100,
	xlabel=$x$,
	ylabel={$\bd_x\rA^{-1}\left(\mathbb{N}_{\geqslant 2}\right)$},
	width=8cm,
    height=6cm,
    clip=false,grid=major,clip marker paths=true,scaled x ticks=false
]
\addplot[color=newblue,mark=x] plot coordinates {(5000., .744800) (6000., .743167) (7000., .743571) (8000., .742750) (9000., .742000) (10000., .741800) (11000., .740545) (12000., .739917) (13000., .740077) (14000., .739214) (15000., .738667) (16000., .739000) (17000., .738353) (18000., .738000) (19000., .737368) (20000., .737250) (21000., .736810) (22000., .736455) (23000., .736261) (24000., .736125) (25000., .736160) (26000., .735885) (27000., .735963) (28000., .735214) (29000., .735103) (30000., .734900) (31000., .734839) (32000., .734438) (33000., .734212) (34000., .734118) (35000., .733943) (36000., .733861) (37000., .734027) (38000., .733842) (39000., .733641) (40000., .733675) (41000., .733537) (42000., .733429) (43000., .733558) (44000., .733250) (45000., .733111) (46000., .733022) (47000., .732809) (48000., .732792) (49000., .732673) (50000., .732480) (51000., .732314) (52000., .732250) (53000., .732189) (54000., .732019) (55000., .731891) (56000., .731768) (57000., .731772) (58000., .731724) (59000., .731593) (60000., .731517) (61000., .731459) (62000., .731323) (63000., .731270) (64000., .731281) (65000., .731015) (66000., .730955) (67000., .730806) (68000., .730897) (69000., .730696) (70000., .730414) (71000., .730366) (72000., .730389) (73000., .730288) (74000., .730243) (75000., .730187) (76000., .730171) (77000., .730078) (78000., .730038) (79000., .730000) (80000., .729938) (81000., .729827) (82000., .729793) (83000., .729699) (84000., .729643) (85000., .729553) (86000., .729430) (87000., .729437) (88000., .729330) (89000., .729112) (90000., .729167) (91000., .729088) (92000., .728924) (93000., .728978) (94000., .728840) (95000., .728684) (96000., .728792) (97000., .728784) (98000., .728633) (99000., .728556) (100000., .728520) (101000., .728426) (102000., .728471) (103000., .728350) (104000., .728221) (105000., .728248) (106000., .728217) (107000., .728140) (108000., .728009) (109000., .727945) (110000., .728000) (111000., .727937) (112000., .727857) (113000., .727788) (114000., .727684) (115000., .727626) (116000., .727595) (117000., .727581) (118000., .727568) (119000., .727437) (120000., .727408) (121000., .727430) (122000., .727418) (123000., .727341) (124000., .727363) (125000., .727320) (126000., .727325) (127000., .727315) (128000., .727312) (129000., .727271) (130000., .727169) (131000., .727153) (132000., .727083) (133000., .727045) (134000., .726963) (135000., .726881) (136000., .726868) (137000., .726854) (138000., .726841) (139000., .726777) (140000., .726771) (141000., .726759) (142000., .726754) (143000., .726685) (144000., .726639) (145000., .726600) (146000., .726562) (147000., .726571) (148000., .726574) (149000., .726497) (150000., .726513) (151000., .726483) (152000., .726428) (153000., .726392) (154000., .726357) (155000., .726297) (156000., .726288) (157000., .726223) (158000., .726196) (159000., .726170) (160000., .726131) (161000., .726106) (162000., .726043) (163000., .726025) (164000., .725988) (165000., .725970) (166000., .725916) (167000., .725862) (168000., .725815) (169000., .725775) (170000., .725747) (171000., .725772) (172000., .725738) (173000., .725757) (174000., .725713) (175000., .725680) (176000., .725636) (177000., .725644) (178000., .725551) (179000., .725503) (180000., .725522) (181000., .725475) (182000., .725412) (183000., .725415) (184000., .725380) (185000., .725351) (186000., .725366) (187000., .725332) (188000., .725324) (189000., .725333) (190000., .725353) (191000., .725309) (192000., .725255) (193000., .725249) (194000., .725237) (195000., .725185) (196000., .725158) (197000., .725107) (198000., .725101) (199000., .725095) (200000., .725025) (201000., .725015) (202000., .725000) (203000., .724980) (204000., .724936) (205000., .724883) (206000., .724864) (207000., .724845) (208000., .724837) (209000., .724799) (210000., .724819) (211000., .724787) (212000., .724764) (213000., .724723) (214000., .724752) (215000., .724749) (216000., .724718) (217000., .724673) (218000., .724642) (219000., .724607) (220000., .724632) (221000., .724602) (222000., .724577) (223000., .724556) (224000., .724491) (225000., .724520) (226000., .724478) (227000., .724458) (228000., .724425) (229000., .724428) (230000., .724435) (231000., .724420) (232000., .724384) (233000., .724382) (234000., .724295) (235000., .724306) (236000., .724246) (237000., .724211) (238000., .724181) (239000., .724192) (240000., .724162) (241000., .724149) (242000., .724157) (243000., .724136) (244000., .724127) (245000., .724135) (246000., .724081) (247000., .724069) (248000., .724028) (249000., .724080) (250000., .724052) (251000., .724020) (252000., .724024) (253000., .724036) (254000., .724016) (255000., .723984) (256000., .723980) (257000., .723930) (258000., .723884) (259000., .723873) (260000., .723862) (261000., .723824) (262000., .723805) (263000., .723783) (264000., .723784) (265000., .723762) (266000., .723759) (267000., .723734) (268000., .723757) (269000., .723710) (270000., .723726) (271000., .723723) (272000., .723706) (273000., .723678) (274000., .723642) (275000., .723636) (276000., .723641) (277000., .723610) (278000., .723561) (279000., .723541) (280000., .723543) (281000., .723544) (282000., .723543) (283000., .723516) (284000., .723468) (285000., .723491) (286000., .723462) (287000., .723422) (288000., .723392) (289000., .723363) (290000., .723359) (291000., .723364) (292000., .723332) (293000., .723338) (294000., .723310) (295000., .723319) (296000., .723307) (297000., .723293) (298000., .723275) (299000., .723231) (300000., .723220) (301000., .723186) (302000., .723149) (303000., .723139) (304000., .723128) (305000., .723148) (306000., .723150) (307000., .723153) (308000., .723110) (309000., .723094) (310000., .723068) (311000., .723051) (312000., .723019) (313000., .722997) (314000., .723000) (315000., .722978) (316000., .722962) (317000., .722965) (318000., .722937) (319000., .722925) (320000., .722925) (321000., .722897) (322000., .722891) (323000., .722876) (324000., .722836) (325000., .722843) (326000., .722828) (327000., .722771) (328000., .722771) (329000., .722723) (330000., .722733) (331000., .722704) (332000., .722699) (333000., .722670) (334000., .722659) (335000., .722642) (336000., .722631) (337000., .722626) (338000., .722606) (339000., .722593) (340000., .722568) (341000., .722551) (342000., .722541) (343000., .722516) (344000., .722506) (345000., .722496) (346000., .722474) (347000., .722458) (348000., .722437) (349000., .722430) (350000., .722429) (351000., .722410) (352000., .722409) (353000., .722394) (354000., .722390) (355000., .722392) (356000., .722379) (357000., .722336) (358000., .722313) (359000., .722304) (360000., .722281) (361000., .722247) (362000., .722229) (363000., .722229) (364000., .722223) (365000., .722216) (366000., .722202) (367000., .722221) (368000., .722212) (369000., .722184) (370000., .722154) (371000., .722151) (372000., .722137) (373000., .722142) (374000., .722115) (375000., .722117) (376000., .722085) (377000., .722077) (378000., .722074) (379000., .722074) (380000., .722074) (381000., .722042) (382000., .722031) (383000., .722003) (384000., .722003) (385000., .721997) (386000., .721997) (387000., .721961) (388000., .721941) (389000., .721915) (390000., .721923) (391000., .721913) (392000., .721888) (393000., .721891) (394000., .721893) (395000., .721878) (396000., .721854) (397000., .721839) (398000., .721807) (399000., .721817) (400000., .721808) (401000., .721786) (402000., .721756) (403000., .721734) (404000., .721715) (405000., .721691) (406000., .721672) (407000., .721644) (408000., .721652) (409000., .721658) (410000., .721654) (411000., .721625) (412000., .721624) (413000., .721605) (414000., .721577) (415000., .721586) (416000., .721579) (417000., .721552) (418000., .721541) (419000., .721556) (420000., .721538) (421000., .721515) (422000., .721502) (423000., .721485) (424000., .721498) (425000., .721492) (426000., .721477) (427000., .721443) (428000., .721409) (429000., .721394) (430000., .721393) (431000., .721369) (432000., .721356) (433000., .721356) (434000., .721353) (435000., .721368) (436000., .721360) (437000., .721336) (438000., .721338) (439000., .721335) (440000., .721339) (441000., .721317) (442000., .721303) (443000., .721302) (444000., .721306) (445000., .721288) (446000., .721262) (447000., .721246) (448000., .721225) (449000., .721205) (450000., .721218) (451000., .721244) (452000., .721239) (453000., .721208) (454000., .721185) (455000., .721149) (456000., .721154) (457000., .721140) (458000., .721122) (459000., .721124) (460000., .721096) (461000., .721100) (462000., .721078) (463000., .721076) (464000., .721073) (465000., .721069) (466000., .721058) (467000., .721049) (468000., .721049) (469000., .721062) (470000., .721055) (471000., .721051) (472000., .721038) (473000., .721019) (474000., .721015) (475000., .721021) (476000., .721013) (477000., .721008) (478000., .720975) (479000., .720981) (480000., .720979) (481000., .720985) (482000., .720973) (483000., .720977) (484000., .720977) (485000., .720959) (486000., .720951) (487000., .720942) (488000., .720930) (489000., .720939) (490000., .720924) (491000., .720904) (492000., .720892) (493000., .720884) (494000., .720883) (495000., .720887) (496000., .720883) (497000., .720863) (498000., .720871) (499000., .720850) (500000., .720844) (501000., .720830) (502000., .720821) (503000., .720805) (504000., .720810) (505000., .720800) (506000., .720794) (507000., .720801) (508000., .720789) (509000., .720780) (510000., .720771) (511000., .720755) (512000., .720758) (513000., .720739) (514000., .720733) (515000., .720730) (516000., .720711) (517000., .720721) (518000., .720726) (519000., .720699) (520000., .720696) (521000., .720679) (522000., .720672) (523000., .720654) (524000., .720635) (525000., .720611) (526000., .720605) (527000., .720607) (528000., .720597) (529000., .720580) (530000., .720566) (531000., .720572) (532000., .720568) (533000., .720568) (534000., .720554) (535000., .720553) (536000., .720554) (537000., .720557) (538000., .720552) (539000., .720531) (540000., .720502) (541000., .720490) (542000., .720483) (543000., .720492) (544000., .720476) (545000., .720453) (546000., .720449) (547000., .720452) (548000., .720443) (549000., .720448) (550000., .720435) (551000., .720432) (552000., .720435) (553000., .720416) (554000., .720426) (555000., .720411) (556000., .720381) (557000., .720393) (558000., .720378) (559000., .720374) (560000., .720379) (561000., .720380) (562000., .720363) (563000., .720366) (564000., .720339) (565000., .720338) (566000., .720325) (567000., .720316) (568000., .720301) (569000., .720288) (570000., .720279) (571000., .720277) (572000., .720269) (573000., .720276) (574000., .720263) (575000., .720249) (576000., .720248) (577000., .720255) (578000., .720246) (579000., .720235) (580000., .720221) (581000., .720201) (582000., .720191) (583000., .720172) (584000., .720178) (585000., .720173) (586000., .720167) (587000., .720164) (588000., .720177) (589000., .720166) (590000., .720139) (591000., .720125) (592000., .720103) (593000., .720098) (594000., .720091) (595000., .720077) (596000., .720060) (597000., .720059) (598000., .720047) (599000., .720040) (600000., .720037) (601000., .720017) (602000., .720008) (603000., .720017) (604000., .720012) (605000., .720010) (606000., .720007) (607000., .719998) (608000., .719987) (609000., .719980) (610000., .719984) (611000., .719975) (612000., .719961) (613000., .719956) (614000., .719953) (615000., .719945) (616000., .719935) (617000., .719938) (618000., .719942) (619000., .719924) (620000., .719921) (621000., .719923) (622000., .719916) (623000., .719907) (624000., .719894) (625000., .719893) (626000., .719885) (627000., .719869) (628000., .719855) (629000., .719833) (630000., .719822) (631000., .719802) (632000., .719799) (633000., .719784) (634000., .719782) (635000., .719791) (636000., .719780) (637000., .719769) (638000., .719757) (639000., .719726) (640000., .719716) (641000., .719716) (642000., .719709) (643000., .719686) (644000., .719677) (645000., .719678) (646000., .719672) (647000., .719665) (648000., .719662) (649000., .719658) (650000., .719651) (651000., .719634) (652000., .719613) (653000., .719608) (654000., .719590) (655000., .719595) (656000., .719590) (657000., .719583) (658000., .719573) (659000., .719580) (660000., .719579) (661000., .719579) (662000., .719566) (663000., .719551) (664000., .719538) (665000., .719537) (666000., .719533) (667000., .719520) (668000., .719506) (669000., .719495) (670000., .719494) (671000., .719490) (672000., .719490) (673000., .719487) (674000., .719482) (675000., .719480) (676000., .719469) (677000., .719462) (678000., .719437) (679000., .719436) (680000., .719440) (681000., .719430) (682000., .719418) (683000., .719411) (684000., .719424) (685000., .719415) (686000., .719415) (687000., .719400) (688000., .719407) (689000., .719406) (690000., .719410) (691000., .719382) (692000., .719380) (693000., .719382) (694000., .719372) (695000., .719355) (696000., .719361) (697000., .719352) (698000., .719345) (699000., .719338) (700000., .719340) (701000., .719328) (702000., .719312) (703000., .719317) (704000., .719318) (705000., .719322) (706000., .719320) (707000., .719322) (708000., .719321) (709000., .719307) (710000., .719300) (711000., .719297) (712000., .719305) (713000., .719293) (714000., .719287) (715000., .719277) (716000., .719277) (717000., .719258) (718000., .719241) (719000., .719231) (720000., .719225) (721000., .719222) (722000., .719229) (723000., .719228) (724000., .719228) (725000., .719228) (726000., .719212) (727000., .719216) (728000., .719205) (729000., .719196) (730000., .719189) (731000., .719176) (732000., .719171) (733000., .719168) (734000., .719163) (735000., .719165) (736000., .719162) (737000., .719151) (738000., .719145) (739000., .719135) (740000., .719134) (741000., .719126) (742000., .719120) (743000., .719120) (744000., .719109) (745000., .719097) (746000., .719083) (747000., .719078) (748000., .719063) (749000., .719055) (750000., .719051) (751000., .719043) (752000., .719051) (753000., .719049) (754000., .719037) (755000., .719042) (756000., .719036) (757000., .719024) (758000., .719020) (759000., .719017) (760000., .719021) (761000., .719022) (762000., .719022) (763000., .719018) (764000., .719024) (765000., .719010) (766000., .719016) (767000., .719016) (768000., .719003) (769000., .719005) (770000., .718996) (771000., .718996) (772000., .718994) (773000., .718983) (774000., .718972) (775000., .718972) (776000., .718960) (777000., .718950) (778000., .718949) (779000., .718952) (780000., .718949) (781000., .718940) (782000., .718936) (783000., .718925) (784000., .718906) (785000., .718899) (786000., .718896) (787000., .718879) (788000., .718871) (789000., .718876) (790000., .718871) (791000., .718874) (792000., .718864) (793000., .718871) (794000., .718861) (795000., .718854) (796000., .718849) (797000., .718842) (798000., .718841) (799000., .718836) (800000., .718825) (801000., .718835) (802000., .718820) (803000., .718812) (804000., .718805) (805000., .718799) (806000., .718800) (807000., .718788) (808000., .718778) (809000., .718765) (810000., .718765) (811000., .718774) (812000., .718777) (813000., .718766) (814000., .718769) (815000., .718767) (816000., .718772) (817000., .718767) (818000., .718764) (819000., .718756) (820000., .718762) (821000., .718759) (822000., .718755) (823000., .718752) (824000., .718750) (825000., .718741) (826000., .718742) (827000., .718745) (828000., .718736) (829000., .718732) (830000., .718725) (831000., .718727) (832000., .718719) (833000., .718726) (834000., .718717) (835000., .718704) (836000., .718700) (837000., .718687) (838000., .718673) (839000., .718669) (840000., .718648) (841000., .718646) (842000., .718644) (843000., .718649) (844000., .718640) (845000., .718643) (846000., .718635) (847000., .718623) (848000., .718630) (849000., .718631) (850000., .718621) (851000., .718616) (852000., .718622) (853000., .718620) (854000., .718607) (855000., .718598) (856000., .718588) (857000., .718590) (858000., .718583) (859000., .718581) (860000., .718576) (861000., .718582) (862000., .718575) (863000., .718581) (864000., .718566) (865000., .718566) (866000., .718568) (867000., .718562) (868000., .718562) (869000., .718552) (870000., .718553) (871000., .718557) (872000., .718541) (873000., .718536) (874000., .718523) (875000., .718525) (876000., .718521) (877000., .718513) (878000., .718505) (879000., .718510) (880000., .718507) (881000., .718509) (882000., .718498) (883000., .718494) (884000., .718485) (885000., .718473) (886000., .718475) (887000., .718467) (888000., .718466) (889000., .718458) (890000., .718448) (891000., .718442) (892000., .718435) (893000., .718428) (894000., .718424) (895000., .718426) (896000., .718424) (897000., .718415) (898000., .718409) (899000., .718405) (900000., .718398) (901000., .718395) (902000., .718381) (903000., .718377) (904000., .718364) (905000., .718367) (906000., .718359) (907000., .718364) (908000., .718355) (909000., .718351) (910000., .718348) (911000., .718329) (912000., .718331) (913000., .718333) (914000., .718313) (915000., .718316) (916000., .718313) (917000., .718304) (918000., .718305) (919000., .718304) (920000., .718293) (921000., .718284) (922000., .718277) (923000., .718272) (924000., .718273) (925000., .718255) (926000., .718262) (927000., .718262) (928000., .718261) (929000., .718248) (930000., .718251) (931000., .718243) (932000., .718240) (933000., .718235) (934000., .718228) (935000., .718235) (936000., .718228) (937000., .718228) (938000., .718217) (939000., .718217) (940000., .718204) (941000., .718202) (942000., .718203) (943000., .718195) (944000., .718188) (945000., .718184) (946000., .718167) (947000., .718158) (948000., .718149) (949000., .718144) (950000., .718141) (951000., .718142) (952000., .718130) (953000., .718133) (954000., .718124) (955000., .718119) (956000., .718113) (957000., .718101) (958000., .718098) (959000., .718094) (960000., .718091) (961000., .718082) (962000., .718080) (963000., .718065) (964000., .718061) (965000., .718055) (966000., .718049) (967000., .718049) (968000., .718036) (969000., .718031) (970000., .718026) (971000., .718029) (972000., .718032) (973000., .718024) (974000., .718007) (975000., .718009) (976000., .718007) (977000., .717997) (978000., .717985) (979000., .717982) (980000., .717979) (981000., .717963) (982000., .717964) (983000., .717955) (984000., .717943) (985000., .717936) (986000., .717934) (987000., .717935) (988000., .717929) (989000., .717925) (990000., .717917) (991000., .717915) (992000., .717920) (993000., .717911) (994000., .717898) (995000., .717894) (996000., .717886) (997000., .717877) (998000., .717886) (999000., .717884) (.100000e7, .717875)};
\addplot[domain=5000:1000000,ultra thick,newpurple] {.7839171430/x^(0.638407522714256e-2)};
\end{axis}
\end{tikzpicture}
\end{tabular}